\documentclass{amsart}
\usepackage{amsmath}
\usepackage{amssymb}
\usepackage{amsthm}

\DeclareMathOperator{\im}{im}
\DeclareMathOperator{\tr}{tr}

\DeclareMathOperator{\Vol}{Vol}
\DeclareMathOperator{\dvol}{dvol}

\DeclareMathOperator{\Ric}{Ric}

\DeclareMathOperator{\Rm}{Rm}

\DeclareMathOperator{\End}{End}

\DeclareMathOperator{\hash}{\sharp}

\DeclareMathOperator{\Imaginary}{Im}

\newcommand{\Diff}{\mathrm{Diff}}
\newcommand{\Con}{\mathrm{Con}}

\newcommand{\oL}{\overline{L}}

\newcommand{\oX}{\overline{X}}

\newcommand{\of}{\overline{f}}
\newcommand{\og}{\overline{g}}

\newcommand{\he}{\widehat{e}}
\newcommand{\hg}{\widehat{g}}

\newcommand{\hQ}{\widehat{Q}}

\newcommand{\lp}{\langle}
\newcommand{\rp}{\rangle}
\newcommand{\lv}{\lvert}
\newcommand{\rv}{\rvert}
\newcommand{\lV}{\lVert}
\newcommand{\rV}{\rVert}
\newcommand{\llp}{\lp\!\lp}
\newcommand{\rrp}{\rp\!\rp}
\newcommand{\leftllp}{\left\lp\!\!\!\left\lp}
\newcommand{\rightrrp}{\right\rp\!\!\!\right\rp}
\newcommand{\bigllp}{\bigl\lp\!\bigl\lp}
\newcommand{\bigrrp}{\bigr\rp\!\bigr\rp}

\newcommand{\bCP}{\mathbb{C}P}

\newcommand{\mC}{\mathcal{C}}

\newcommand{\mE}{\mathcal{E}}
\newcommand{\mF}{\mathcal{F}}

\newcommand{\mI}{\mathcal{I}}

\newcommand{\mK}{\mathcal{K}}
\newcommand{\mL}{\mathcal{L}}
\newcommand{\mM}{\mathcal{M}}

\newcommand{\mP}{\mathcal{P}}

\newcommand{\mR}{\mathcal{R}}
\newcommand{\mS}{\mathcal{S}}

\newcommand{\mV}{\mathcal{V}}

\newcommand{\kC}{\mathfrak{C}}

\newcommand{\kK}{\mathfrak{K}}

\newcommand{\kM}{\mathfrak{M}}

\newcommand{\bC}{\mathbb{C}}

\newcommand{\bN}{\mathbb{N}}

\newcommand{\bR}{\mathbb{R}}

\def\sideremark#1{\ifvmode\leavevmode\fi\vadjust{\vbox to0pt{\vss
 \hbox to 0pt{\hskip\hsize\hskip1em
 \vbox{\hsize3cm\tiny\raggedright\pretolerance10000
 \noindent #1\hfill}\hss}\vbox to8pt{\vfil}\vss}}}

\newcommand{\suchthat}{\mathrel{}\middle|\mathrel{}}

\newcommand{\comment}[1]{}

\newtheorem{thm}{Theorem}[section]
\newtheorem{prop}[thm]{Proposition}
\newtheorem{lem}[thm]{Lemma}
\newtheorem{cor}[thm]{Corollary}

\theoremstyle{definition}
\newtheorem{defn}[thm]{Definition}

\newtheorem{example}{Example}[subsection]

\theoremstyle{remark}
\newtheorem{remark}[thm]{Remark}

\numberwithin{equation}{section}

\usepackage{mathtools}

\begin{document}

\title{Conformally variational Riemannian invariants}
\author{Jeffrey S. Case}
\address{109 McAllister Building \\ Penn State University \\ University Park, PA 16802 \\ USA}
\email{jscase@psu.edu}
\author{Yueh-Ju Lin}
\address{Department of Mathematics \\ Princeton University \\ Princeton, NJ 08544 \\ USA}
\email{yuehjul@math.princeton.edu}
\author{Wei Yuan}
\address{Department of Mathematics \\ Sun Yat-sen University \\ Guangzhou, Guangdong
510275 \\ China}
\email{gnr-x@163.com}
\begin{abstract}
 Conformally variational Riemannian invariants (CVIs), such as the scalar curvature, are homogeneous scalar invariants which arise as the gradient of a Riemannian functional.  We establish a wide range of stability and rigidity results involving CVIs, generalizing many such results for the scalar curvature.
\end{abstract}
\maketitle

\section{Introduction}
\label{sec:intro}

The variational properties of the scalar curvature underlie its outsize importance in Riemannian geometry.  Among the consequences of these properties are variational characterizations of Einstein metrics~\cite{Besse}, constant scalar curvature metrics~\cite{Yamabe1960}, and static and related metrics~\cite{Corvino2000,CorvinoEichmairMiao2013}; obstructions to the existence of metrics with prescribed scalar curvature in a conformal class~\cite{Bourguignon1985,KazdanWarner1975} and existence results within the space of Riemannian metrics~\cite{FischerMarsden1975}; the Schur lemma~\cite{Kazdan1981} and the almost Schur lemma~\cite{DeLellisTopping2012};
and rigidity of flat manifolds~\cite{FischerMarsden1975}.  Indeed, these results are all consequences of readily checked properties the first and second variations of the scalar curvature.  In this article, we show that there are large classes of scalar Riemannian invariants which satisfy similar variational properties and consequences thereof.  This work is heavily influenced by previous work of Gover and {\O}rsted~\cite{GoverOrsted2010} which explored obstructions to prescribing the first such class of scalar invariants, conformally variational invariants, within a conformal class.  It is also strongly influenced by recent generalizations of some of the above results to the $Q$-curvature~\cite{ChangGurskyYang1996,LinYuan2015a,LinYuan2015b} and the renormalized volume coefficients~\cite{ChangFang2008,ChangFangGraham2012,Graham2009}.

A \emph{conformally variational Riemannian invariant (CVI)} is a natural Riemannian scalar invariant which is homogeneous and has a conformal primitive.  We say that $L$ is \emph{homogeneous of degree $-2k$} if $L_{c^2g}=c^{-2k}L_g$ for all metrics $g$ and all constants $c>0$, while a \emph{conformal primitive} is a Riemannian functional $\mS$ such that
\[ \left.\frac{d}{dt}\right|_{t=0}\mS\left(e^{2t\Upsilon}g\right) = \int_M L_g\Upsilon\,\dvol_g \]
for all metrics $g$ and all smooth functions $\Upsilon$ of compact support on $M$.  For technical reasons, we consider only CVIs of weight $-2k$ on manifolds of dimension $n\geq2k$; see Section~\ref{sec:bg} for further discussion.  For practical purposes, it is easier to work with the equivalent definition that $L$ is a CVI of weight $-2k$ if and only if there is a natural formally self-adjoint operator $A_g$ such that $A_g(1)=0$ and
\begin{equation}
 \label{eqn:intro_cvi_linearization}
 \left.\frac{\partial}{\partial t}\right|_{t=0} L_{e^{2t\Upsilon}g} = -2kL_g\Upsilon + A_g(\Upsilon)
\end{equation}
for all metrics $g$ and all smooth functions $\Upsilon$; see~\cite{BransonGover2008}.  Since the scalar curvature is the conformal gradient of the total scalar curvature functional in dimensions at least three~\cite{Yamabe1960} and the conformal gradient of the functional determinant of the Laplacian in dimension two~\cite{OsgoodPhillipsSarnak1988,Polyakov1981}, it is a CVI of weight $-2$; this can also be checked directly by computing the conformal linearization~\eqref{eqn:intro_cvi_linearization} of the scalar curvature.

The first systematic study of CVIs was carried out by Branson and {\O}rsted~\cite{BransonOrsted1991}, where they observed that the fiberwise heat trace coefficients of conformally covariant operators are CVIs.  This observation plays a key role in their computations of certain functional determinants in dimension four and six~\cite{Branson1995,Branson1996,BransonOrsted1991b}.  Indeed, these computations allow one to show that every CVI of weight $-4$ has a conformal primitive in dimension $4$ which is the functional determinant of a product of conformally covariant operators.

Since CVIs have a conformal primitive, one wonders about using variational methods to prescribe a given CVI.  Within a conformal class, this problem is obstructed: Gover and {\O}rsted showed~\cite{GoverOrsted2010} (see also~\cite{Bourguignon1985,DelanoeRobert2007}) that the Kazdan--Warner obstruction extends to CVIs.  For CVIs $L$ of weight $-2k$ on manifolds of dimension $n>2k$, this is a consequence of the fact that a nonzero multiple of the total $L$-functional, $\mS(g):=\int L_g\,\dvol_g$, is a conformal primitive for $L$.  Indeed, let $\Gamma\colon S^2T^\ast M\to C^\infty(M)$ be the metric linearization of $L$ and set $S=-\Gamma^\ast(1)$ for $\Gamma^\ast$ the $L^2$-dual of $\Gamma$.  It follows immediately that
\[ \left.\frac{\partial}{\partial t}\right|_{t=0}\mS(g+th) = -\int_M \bigl\lp S-\frac{L}{2}g,h\bigr\rp\,\dvol_g , \]
while naturality implies that $S-\frac{L}{2}g$ is divergence-free.  The trace of $S$ can be computed by letting $h=\Upsilon g$, yielding a Schur-type lemma (see Lemma~\ref{lem:divGamma}).  More generally, we have the following ``almost Schur lemma'' (cf.\ \cite{DeLellisTopping2012,LinYuan2015b}):

\begin{thm}
 \label{thm:intro_almost_schur}
 Let $L$ be a CVI of weight $-2k$.  Let $(M^n,g)$ be a compact Riemannian manifold with $n>2k$ and $\Ric\geq0$.  Then
 \begin{equation}
  \label{eqn:intro_almost_schur}
  \int_M \left(L-\oL\right)^2\dvol \leq \frac{4n(n-1)}{(n-2k)^2}\int_M \lv S_0\rv^2\,\dvol,
 \end{equation}
 where $\oL=(\int L)/(\int 1)$ is the average of $L$ and $S_0$ is the trace-free part of $S$.
\end{thm}

See Section~\ref{sec:schur} for a proof and discussion of the case of equality in~\eqref{eqn:intro_almost_schur}.

To establish affirmative results for the problem of prescribing the value of a CVI, we impose an ellipticity assumption: A CVI $L$ is \emph{elliptic} if the operator $A$ determined by~\eqref{eqn:intro_cvi_linearization} is elliptic.  For elliptic CVIs, the problem of locally prescribing $L$ is unobstructed if $\ker\Gamma^\ast=\{0\}$:

\begin{thm}
 \label{thm:intro_corollary}
 Let $L$ be a CVI which is elliptic at $(M^n,g)$.  If $\ker\Gamma^\ast=\{0\}$, then there are neighborhoods $U$ of $g$ in the space of Riemannian metrics on $M$ and $V$ of $L_g$ in $C^\infty(M)$ such that for any $f\in V$, there is a metric $\hg\in U$ such that $L_{\hg}=f$.
\end{thm}

Theorem~\ref{thm:intro_corollary} was proven in the special cases of the scalar curvature and the fourth-order $Q$-curvature by Fischer and Marsden~\cite{FischerMarsden1975} and the second- and third-named authors~\cite{LinYuan2015a}, respectively.  See Section~\ref{sec:singular} for a description of the topologies used on the spaces of metrics and smooth functions on $M$.

When $L$ is the scalar curvature, the assumption $\ker\Gamma^\ast=\{0\}$ in Theorem~\ref{thm:intro_corollary} is equivalent to the assumption that $(M^n,g)$ is not a static metric in the sense of Corvino~\cite{Corvino2000}.  Similar to Theorem~\ref{thm:intro_corollary}, we prove in Section~\ref{sec:critical} that the problem of locally prescribing $L$ and $\Vol(M)$ is unobstructed if there are no solutions to $\Gamma^\ast(f)=\lambda g$ (see~\cite{CorvinoEichmairMiao2013} for the case when $L$ is the scalar curvature).  As in the case of the scalar curvature~\cite{MiaoTam2008} and the fourth-order $Q$-curvature~\cite{LinYuan2015b}, the assumption $\ker\Gamma^\ast=\{0\}$ is rather weak:

\begin{thm}
 \label{thm:intro_Lconst}
 Let $L$ be a CVI of weight $-2k$.  Suppose that $(M^n,g)$ is such that $L$ is elliptic at $g$.  Assume further that $DL_g$ is of order at most four or that $k\leq3$.  If $L_g$ is nonconstant, then $\ker\Gamma^\ast=\{0\}$.
\end{thm}

The reason for the restriction $k\leq3$ in Theorem~\ref{thm:intro_Lconst} is that our proof requires $DL_g$ to satisfy the weak unique continuation property; see Section~\ref{sec:singular} for further discussion.

One nice property of the total scalar curvature functional and the functional determinant of the Laplacian on surfaces --- the \emph{standard} conformal primitives for the scalar curvature --- is that they have one-sided bounds in a conformal class and are extremized at an Einstein metric in a conformal class, if it exists~\cite{Obata1971,OsgoodPhillipsSarnak1988,Yamabe1960}.  It is also the case that the functional determinants of the conformal Laplacian and the Dirac operator are extremized at round metrics on the four-sphere in their conformal class~\cite{BransonChangYang1992}.  However, this property is not satisfied by the functional determinant of the Paneitz operator on the four-sphere; indeed, this functional is neither bounded above nor below and admits critical points which are not round~\cite{GurskyMalchiodi2012}.  We seek sufficient conditions for a CVI to be extremized at Einstein metrics in their conformal class, when they exist.

Motivated by the above problem, we say that a CVI $L$ of weight $-2k$ is \emph{variationally stable at $(M^n,g)$} if $L_g$ is constant and
\[ \int_M \Upsilon\,A_g(\Upsilon)\,\dvol_g \geq 2kL_g\int_M \Upsilon^2\,\dvol_g \]
for all $\Upsilon\in C^\infty(M)$, with equality if and only if there is a vector field $X$ such that $\mL_Xg=\Upsilon g$.  Here $A_g$ is the operator from~\eqref{eqn:intro_cvi_linearization}.  As we discuss in Section~\ref{sec:second_derivative}, if $n>2k$, then these assumptions are equivalent to $D^2\mS$ being positive when restricted to curves of conformal metrics with fixed volume which are transverse to the action of the conformal group.

Variational stability and ellipticity together have interesting implications for the problem of prescribing $L$.  First, in Proposition~\ref{prop:local_stability} we show that these properties imply that, in non-critical dimensions, the total $L$-curvature functional is locally rigid at a metric $g$ at which $L$ is variationally stable and elliptic.  Second, if $L$ is constant, elliptic, and variationally stable, and if there is a nontrivial solution to $\Gamma^\ast(f)=\lambda g$, then the only obstruction to $g$ being a critial point of the total $L$-curvature functional is given by conformal Killing fields.

\begin{thm}
 \label{thm:intro_nontrivial-Lsingular}
 Let $L$ be a CVI of weight $-2k$.  Let $(M^n,g)$ admit a nontrivial $f\in C^\infty(M)$ such that $\Gamma^\ast(f)=\lambda g$ for some $\lambda\in\bR$.  If $L$ is constant, elliptic, and variationally stable at $g$ and if $(M^n,g)$ does not admit a conformal Killing field which is not Killing, then $f$ is constant.  In particular, if $n\not=2k$, then $(M^n,g)$ is a critical point of the total $L$-curvature functional $g\mapsto\int L_g\dvol_g$ in the space of Riemannian metrics of fixed volume.
\end{thm}

However, variational stability and ellipticity cannot be sufficient conditions for the standard conformal primitive to be extremized at any Einstein metric in its conformal class.  Indeed, the conformal gradient of the functional determinant of the Paneitz operator on four-manifolds~\cite{Branson1996,GurskyMalchiodi2012} provides a counterexample; see Section~\ref{sec:weight} for further discussion.

Our last general observation about CVIs is another rigidity result, this time for flat metrics.  It is well-known that the scalar curvature is rigid at flat metrics, in that if $g$ is close to a flat metric and has nonnegative scalar curvature, it must be flat~\cite{FischerMarsden1975}.  Indeed, this is true globally~\cite{GromovLawson1980,GromovLawson1983,SchoenYau1979_torus,SchoenYau1979s}.  A similar local result is known for the fourth-order $Q$-curvature~\cite{LinYuan2015a}.  We observe that this property follows for any CVI under a simple spectral assumption; indeed, this result only requires the homogeneity and naturality of $L$.

\begin{thm}
 \label{thm:intro_rigidity}
 Let $L$ be a natural Riemannian scalar invariant which is homogeneous.  Assume additionally that for every flat manifold $(M^n,g)$, it holds that
 \begin{equation}
  \label{eqn:intro_rigidity}
  \int_M \left.\frac{\partial^2}{\partial t^2}\right|_{t=0}\left(L_{g+th}\right)\dvol_g \leq 0
 \end{equation}
 for all $h\in\ker\delta$, with equality if and only if $h$ is parallel.  Then for any flat manifold $(M^n,g)$, there is a neighborhood $U$ of $g$ in the space of metrics such that if $\og\in U$ satisfies $L_{\og}\geq0$, then $\og$ is flat.
\end{thm}

When $L$ is a CVI of weight $-2k$, the condition that~\eqref{eqn:intro_rigidity} holds at all flat manifolds implies that $L$ must be a nonzero multiple of the $Q$-curvature of order $2k$, up to the addition of a CVI of weight $-2k$ of order at most $2k-2$ in the metric; see Lemma~\ref{lem:rigid_implies_linear_term} for details.  Remark~\ref{rk:rigid_requires_cvi} gives an example of an invariant which satisfies the hypotheses of Theorem~\ref{thm:intro_rigidity} but is not a CVI.

The crucial observation is that the assumptions of Theorem~\ref{thm:intro_rigidity} are enough to imply the following spectral gap: There is a constant $C>0$ such that
\[ \int_M \left.\frac{\partial^2}{\partial t^2}\right|_{t=0}\left(L_{g+th}\right)\dvol_g \leq -C\int_M \lv\nabla^k h\rv^2 \]
for all $h\in\ker\delta$, where $-2k$ is the weight of $L$ and $C>0$.

The results of this article depend only on being able to compute the first and second conformal variations of natural Riemannian scalar invariants and, in the case of Section~\ref{sec:rigidity}, the second metric variation at flat metrics.  These computations are all straightforward.  Moreover, in many cases of interest --- for example, at Einstein metrics --- these computations are drastically simplified.  We illustrate this in Section~\ref{sec:weight} via a thorough discussion of examples through two perspectives.  First, we recall the definitions of the renormalized volume coefficients and the $Q$-curvatures via Poincar\'e metrics~\cite{FeffermanGraham2012} and recall the fact that these invariants are all variationally stable at Einstein metrics with positive scalar curvature (see also~\cite{ChangFangGraham2012}).  Second, we give bases for the spaces of CVIs of weight $-2$, $-4$, and $-6$ with enough detail to check which are variationally stable at Einstein metrics with positive scalar curvature.  These examples show the wide applicability of our results to variational problems in Riemannian geometry.

This article is organized as follows:

In Section~\ref{sec:bg}, we make precise the notion of a CVI and describe an algorithm due to Branson~\cite{Branson1995} for finding a basis of CVIs using Weyl's invariant theory.

In Section~\ref{sec:schur}, we use the standard conformal primitive of a CVI to prove Theorem~\ref{thm:intro_almost_schur} and related results.

In Section~\ref{sec:second_derivative}, we introduce the notion of variational stability for a CVI $L$ and explain its importance for the total $L$-curvature functional.  This includes a discussion of the role of the space of conformal Killing fields in the study of variational properties of CVIs.

In Section~\ref{sec:singular}, we introduce the notion of ellipticity for a CVI $L$ and prove a number of results for elliptic, variationally stable CVIs, including Theorem~\ref{thm:intro_corollary}, Theorem~\ref{thm:intro_Lconst} and Theorem~\ref{thm:intro_nontrivial-Lsingular}.

In Section~\ref{sec:critical}, we describe the variational significance of the solutions of $\Gamma^\ast (f)=\lambda g$ for a CVI $L$.  In particular, we prove the aforementioned generalization of Theorem~\ref{thm:intro_corollary}.

In Section~\ref{sec:rigidity}, we study CVIs in neighborhoods of flat metrics.  In particular, we prove Theorem~\ref{thm:intro_rigidity}.

In Section~\ref{sec:weight}, we discuss a number of examples of CVIs.  In Subsection~\ref{subsec:std} we describe the renormalized volume coefficients and the $Q$-curvatures of all orders.  In Subsection~\ref{subsec:weight2} we show how our methods recover the well-known facts that the scalar curvature is elliptic and variationally stable at Einstein metrics.  In Subsection~\ref{subsec:weight4} we give a basis for the space of CVIs of weight $-4$ and describe the cones of CVIs which are elliptic and variationally stable at Einstein metrics with positive scalar curvature.  In Subsection~\ref{subsec:weight6} we give a basis for the space of CVIs of weight $-6$ and check which basis elements are elliptic and variationally stable at Einstein metrics with positive scalar curvature.

\section{Conformally variational invariants}
\label{sec:bg}

Let $M^n$ be a smooth manifold and let $\kM$ denote the space of Riemannian metrics on $M$.  For each $g\in\kM$, we identify the formal tangent space $T_g\kM$ of $\kM$ at $g$ with $S^2T^\ast M$ by identifying $h\in S^2T^\ast M$ with the tangent vector to the curve $t\mapsto g+th$ at $t=0$.  Using this identification, we equip $T_g\kM$ with the inner product
\[ \llp h_1,h_2\rrp = \int_M \lp h_1,h_2\rp_g\,\dvol_g \]
for $h_1,h_2\in T_g\kM$, where $\lp\cdot,\cdot\rp_g$ and $\dvol_g$ are the inner product on $S^2T^\ast M$ and the Riemannian volume element, respectively, induced by $g$.

A \emph{conformal class} is an equivalence class $\kC\subset\kM$ of metrics with respect to the equivalence relation $g_1\sim g_2$ if and only if $g_2=e^{2\Upsilon}g_1$ for some $\Upsilon\in C^\infty(M)$.  We regard $T_g\kC\cong C^\infty(M)$ by identifying $\Upsilon\in C^\infty(M)$ with the tangent vector to the curve $t\mapsto e^{2t\Upsilon}g$ at $t=0$.  The inner product on $T_g\kM$ induces the inner product
\[ \llp \Upsilon_1,\Upsilon_2 \rrp = \int_M \Upsilon_1\Upsilon_2\,\dvol_g \]
for all $\Upsilon_1,\Upsilon_2\in T_g\kC$.

A \emph{natural Riemannian scalar invariant} is a function $L\colon\kM\to C^\infty(M)$ which is both polynomial in the finite jets of $g$ and is invariant under the action of the diffeomorphism group $\Diff(M)$ of $M$ on $\kM$:
\[ L(\phi^\ast g) = \phi^\ast L(g) \]
for all $g\in\kM$ and all $\phi\in\Diff(M)$.  This implies that $L$ is a polynomial in $g$, $g^{-1}$, and the covariant derivatives of the Riemann curvature tensor.  Similarly, a \emph{Riemannian functional} is a function $\mS\colon\kM\to\bR$ which is invariant under the action of $\Diff(M)$.  Also, a \emph{natural Riemannian operator} is both invariant under the action of $\Diff(M)$ and a polynomial in $g$, $g^{-1}$, the Riemann curvature tensor, and the Levi-Civita connection.  We also use the notation $L_g$ to denote the scalar invariant $L(g)$; when the metric $g$ is clear by context, we will simply write $L$ for $L(g)$.

A natural Riemannian scalar invariant $L$ is \emph{homogeneous of degree $-2k$} if
\[ L(c^2g) = c^{-2k} L(g) \]
for all $g\in\kM$ and all $c>0$.  For example, the scalar curvature is homogeneous of degree $-2$.

In terms of the identification $T_g\kM\cong S^2T^\ast M$, the \emph{metric linearization} of the natural Riemannian scalar invariant $L$ at $g\in\kM$ is
\[ DL[h] := \left.\frac{\partial}{\partial t}\right|_{t=0} L(g+th) . \]
Naturality enables us to evaluate the metric linearization $DL$ in the direction of a Lie derivative:

\begin{lem}
 \label{lem:dL_Lie_derivative}
 Let $L$ be a natural Riemannian scalar invariant.  On any Riemannian manifold $(M^n,g)$, it holds that
 \begin{equation}
  \label{eqn:dL_Lie_derivative}
  DL[\mL_Xg] = dL(X)
 \end{equation}
 for all vector fields $X$ on $M$, where $\mL_X$ denotes the Lie derivative.
\end{lem}

\begin{proof}
 Let $X$ be a smooth vector field on $M$ and let $\psi_t\colon M\to M$ be a one-parameter family of diffeomorphisms such that $\psi_0$ is the identity and $\left.\frac{\partial}{\partial t}\right|_{t=0}\psi_t=X$.  Then
 \[ DL[\mL_Xg] = \left.\frac{\partial}{\partial t}\right|_{t=0}L(\psi_t^\ast g) = \left.\frac{\partial}{\partial t}\right|_{t=0}\psi_t^\ast L = dL(X) . \qedhere \]
\end{proof}

In terms of the identification $T_g\kC\cong C^\infty(M)$, the \emph{conformal linearization} of the natural Riemannian scalar invariant $L$ at $g\in\kC$ is
\[ DL(\Upsilon) := \left.\frac{\partial}{\partial t}\right|_{t=0} L(e^{2t\Upsilon}g) . \]
Note that, in terms of the metric linearization of $L$,
\begin{equation}
 \label{eqn:relate_derivatives}
 DL(\Upsilon) = DL[2\Upsilon g] .
\end{equation}
It is for this reason we have opted to use such similar notation for the metric linearization and the conformal linearization.  We shall differentiate between the two linearizations by always using square (resp.\ round) brackets and Latin (resp.\ Greek) characters for the metric (resp.\ conformal) linearization of a natural Riemannian invariant.

A natural Riemannian scalar invariant $L$ is \emph{conformally variational in a conformal class $\kC\subset\kM$} if there is a Riemannian functional $\mS\colon\kM\to\bR$ such that
\begin{equation}
 \label{eqn:conformally_variational_definition}
 \left.\frac{d}{dt}\right|_{t=0} \mS\left(e^{2t\Upsilon}g\right) = \int_M L\Upsilon\,\dvol_g
\end{equation}
for all $g\in\kC$ and all $\Upsilon\in C^\infty(M)$, where $\dvol_g$ denotes the Riemannian volume element determined by $g$.  We say that $L$ is \emph{conformally variational} if it is conformally variational in all conformal classes on which it is defined; note that some invariants (e.g.\ the $Q$-curvatures~\cite{GoverHirachi2004,Graham1992}) are only defined in sufficiently large dimensions.

Our definitions of the metric and the conformal linearizations of a natural Riemannian scalar invariant readily extend to Riemannian functionals.  In this way, we write~\eqref{eqn:conformally_variational_definition} as
\[ D\mS(\Upsilon) = \llp L, \Upsilon\rrp; \]
i.e.\ conformally variational Riemannian scalar invariants are gradients of natural Riemannian functionals.  Indeed, a natural Riemannian scalar invariant $L$ is conformally variational if and only if its conformal linearization is formally self-adjoint with respect to $\llp\cdot,\cdot\rrp$; see~\cite{BransonGover2008}.  In particular, a natural Riemannian invariant which is homogeneous of degree $-2k$ is conformally variational if and only if there is a natural differential operator $T\colon\Omega^1M\to\Omega^1M$ which is symmetric on closed forms and such that
\begin{equation}
 \label{eqn:cvi_linearization}
 DL(\Upsilon) = -2kL\Upsilon - \delta\left(T(d\Upsilon)\right)
\end{equation}
for all $\Upsilon\in T_g\kC$ and all $g\in\kC$, where $\delta\omega=\tr\nabla\omega$ is the divergence operator on one-forms.

A \emph{conformally variational invariant (CVI)} is a conformally variational Riemannian scalar invariant which is homogeneous.  The \emph{weight} of a CVI is its degree of homogeneity.

For the remainder of this article, we reserve special notation for three objects related to the linearization of a CVI $L$.  First, we denote by $\Gamma\colon T_g\kM\to C^\infty(M)$ the linearization of $L$ at $g\in\kM$.  Second, we denote by $\Gamma^\ast\colon C^\infty(M)\to T_g\kM$ the adjoint of $\Gamma$ with respect to the inner product $\llp\cdot,\cdot\rrp$ on $T_g\kM$:
\[ \int_M \lp \Gamma^\ast(f), h\rp_g\,\dvol_g = \int_M f\,\Gamma[h]\,\dvol_g \]
for all $f\in C^\infty(M)$ and all $h\in S^2T^\ast M$.  Third, we denote by $S\in S^2T^\ast M$ the associated symmetric tensor $S=-\Gamma^\ast(1)$.

One immediate consequence of~\eqref{eqn:cvi_linearization} is that for any CVI of weight $-2k$ on a closed $n$-manifold, the functional $\mS(g):=\int L_g\,\dvol_g$ is such that
\begin{equation}
 \label{eqn:first_conformal_derivative}
 D\mS(\Upsilon) = (n-2k)\int_M L_g\Upsilon\,\dvol_g .
\end{equation}
In particular, if $n\not=2k$, then $L$ is the conformal gradient of $\frac{1}{n-2k}\mS$.  On the other hand, since $D\dvol_g[h]=\frac{1}{2}\tr_g h$, we see that
\begin{equation}
 \label{eqn:first_derivative}
 D\mS[h] = - \leftllp S - \frac{1}{2}Lg, h \rightrrp .
\end{equation}
In other words, $S-\frac{1}{2}Lg$ is the gradient of $\mS$.  This observation facilitates a close comparison between our approach to the study of conformally variational Riemannian invariants and the approach taken by Gover and {\O}rsted~\cite{GoverOrsted2010}; see Section~\ref{sec:schur} and Remark~\ref{rk:kazdan_warner} for further comments on this comparison.

\begin{remark}
 If $M^n$ is a closed manifold and $L$ is a CVI of weight $-n$, then the functional $g\mapsto\int L_g\dvol_g$ is invariant in conformal classes.  Instead, one can realize $L$ as a conformal gradient in the following way (cf.\ \cite{BrendleViaclovsky2004}): Let $\kC$ be a conformal class on $M$ and fix a background metric $g_0\in\kC$.  Define $\mS\colon\kC\to\bR$ by
 \begin{equation}
  \label{eqn:associated_lagrangian_critical_dimension}
  \mS(g) = \int_0^1 \int_M u\,L(e^{2su}g_0)\dvol_{e^{2su}g_0}\,ds,
 \end{equation}
 where $g=e^{2u}g_0$.  It is straightward to check from~\eqref{eqn:cvi_linearization} that $L$ is the conformal gradient of $\mS$.
\end{remark}

We conclude this section by discussing both an important example of a CVI and a general construction of CVIs.  This construction is carried out explicitly in Section~\ref{sec:weight} to describe the vector spaces of CVIs of weight $-2$, $-4$, and $-6$.

First, the simplest and most studied example of a CVI is the scalar curvature $R$.  It is well-known (e.g.\ \cite[Section~1.J]{Besse}) that
\begin{equation}
 \label{eqn:full_linearization_R}
 DR[h] = -\lp\Ric,h\rp + \delta^2h - \Delta\tr h .
\end{equation}
In particular, we compute that
\begin{equation}
 \label{eqn:conformal_linearization_R}
 DR(\Upsilon) = -2R\Upsilon - 2(n-1)\Delta\Upsilon,
\end{equation}
and hence $R$ is a CVI of weight $-2$.  From~\eqref{eqn:full_linearization_R} we also recover the familiar formula
\[ \Gamma^\ast(f) = -f\Ric + \nabla^2f - \Delta f\,g \]
for the adjoint of the linearization of $R$.  In particular, we see that $S=\Ric$.  The map $\Gamma^\ast$ and the tensor $\Ric$ both play an important role in studying variational problems involving the scalar curvature (cf.\ \cite{Besse,Corvino2000,FischerMarsden1975}).

Second, we restrict our attention to the classes $\mI_{2k}^n$ of CVIs of weight $-2k$ which are defined on any Riemannian manifold of dimension $n\geq 2k$.  This restriction ensures that Riemannian invariants defined via the Fefferman--Graham ambient metric~\cite{FeffermanGraham2012} are well-defined.  By Weyl's invariant theory~\cite{Epstein1975,Stredder1975,Weyl1997}, the vector space $\mR_{2k}^n$ of scalar Riemannian invariants of weight $-2k$ defined on any Riemannian manifold of dimension $n\geq 2k$ is spanned by complete contractions
\begin{equation}
 \label{eqn:invariant_thy_basis}
 \mathrm{contr}\left(\nabla^{r_1}\Rm\otimes\dotsm\otimes\nabla^{r_j}\Rm\right),
\end{equation}
where $r:=\sum r_j$ is even and $r+2j=2k$.  In particular, $\mR_{2k}^n=\{0\}$ unless $k$ is a nonnegative integer, and $\mR_0^n\cong\bR$ consists of the constant functions.  We say that $L$ is a \emph{nontrivial CVI} if it is homogeneous of weight $-2k$, $k\in\bN$.

Consider the case $n>2k$ and let $\im\delta\subset\mR_{2k}^n$ denote the subspace of exact divergences.  Since $n>2k$, we conclude from~\eqref{eqn:first_conformal_derivative} that $\mR_{2k}^n/\im\delta$ is isomorphic to $\mI_{2k}^n$.  Indeed, the map $\Phi\colon\mI_{2k}^n\to\mR_{2k}^n/\im\delta$ given by $\Phi(L)=[L]$ is an isomorphism with inverse $\Phi^{-1}([L])=L^\prime$, where $L^\prime$ is the conformal gradient of the functional $g\mapsto\frac{1}{n-2k}\int L$.  This provides an explicit method to identify $\mI_{2k}^n$ from a basis of $\mR_{2k}^n/\im\delta$ on Riemannian $n$-manifolds with $n>2k$.

To handle the case $n=2k$, recall that Gilkey showed~\cite{Gilkey1984} that $\mR_{2k}^{2k}$ is isomorphic to $\mR_{2k}^n$ for all $n\geq 2k$, with the restriction map induced by the inclusion of $M^{2k}$ in the Riemannian product $M^{2k}\times T^{n-2k}$ with a flat torus as an explicit isometry.  Thus, by choosing a suitable basis for $\mR_{2k}^n/\im\delta$ above, one can analytically continue in the dimension to identify a basis for $\mI_{2k}^{2k}$; see~\cite{Branson1995} for further discussion.  For our purposes, where we only give explicit bases for $\mI_{2k}^n$, $k\in\{1,2,3\}$ and $n\geq 2k$, it is simple to directly check that this argument provides a basis for $\mI_{2k}^{2k}$ without a precise formulation of analytic continuation in the dimension.  These bases are derived in Section~\ref{sec:weight}.

\section{Schur-type results for CVIs}
\label{sec:schur}

Two well-known relationships between the Ricci curvature $\Ric$ and the scalar curvature $R$ are the trace identity $\tr\Ric=R$ and the consequence $\delta\Ric=\frac{1}{2}dR$ of the second Bianchi identity.  These readily yield Schur's lemma: If $n\geq3$ and $(M^n,g)$ is such that $\Ric=\lambda g$, then $\lambda$ is constant.  More recently, De Lellis and Topping~\cite{DeLellisTopping2012} proved an $L^2$-version of this result under the name ``almost Schur lemma'' for manifolds with nonnegative Ricci curvature; see~\cite{GeWang2012a,GeWang2012b} and~\cite{Kwong2015} for analogues for manifolds with nonnegative scalar curvature and Ricci curvature bounded below, respectively.  The almost Schur lemma was generalized to the fourth-order $Q$-curvature and its associated tensor $S$ by the second- and third-named authors~\cite{LinYuan2015b}.  In this section we show that their result generalizes to any CVI; note that the observations leading to the analogue of Schur's lemma have already been pointed out by Gover and {\O}rsted~\cite{GoverOrsted2010}.

We begin with two basic observations relating a CVI $L$ to the trace and divergence of the associated operator $\Gamma^\ast$ (see~Section~\ref{sec:bg}).

\begin{lem}
 \label{lem:trGamma}
 Let $L$ be a CVI of weight $-2k$.  Then
 \[ \tr \Gamma^\ast(f) = -kLf - \frac{1}{2}\delta\left(T(df)\right) \]
 for all $f\in C^\infty(M)$, where $T$ is as in~\eqref{eqn:cvi_linearization}.
\end{lem}

\begin{proof}
 Let $f,u\in C^\infty(M)$ with $u$ compactly supported.  We compute that
 \[ \llp \Gamma^\ast(f), ug\rrp = \llp f, \Gamma[ug]\rrp = \frac{1}{2}\llp f, DL(u)\rrp = \frac{1}{2}\llp u, DL(f)\rrp , \]
 where the last identity follows from the fact that $DL$ is self-adjoint.  The conclusion now follows from the expression~\eqref{eqn:cvi_linearization} for the conformal linearization of a CVI.
\end{proof}

\begin{lem}
 \label{lem:divGamma}
 Let $L$ be a CVI.  Then
 \[ \delta\Gamma^\ast(f) = -\frac{1}{2}f\,dL \]
 for all $f\in C^\infty(M)$.
\end{lem}

\begin{proof}
 Let $X$ be a compactly-supported smooth vector field on $M$.  Since the formal adjoint of the divergence $\delta$ is $-\frac{1}{2}\mL_Xg$, we observe that
 \begin{multline*}
  \int_M\lp X,\delta\Gamma^\ast(f)\rp,\dvol = -\frac{1}{2}\int_M\lp\mL_Xg,\Gamma^\ast (f)\rp,\dvol \\ = -\frac{1}{2}\int_M f\,\Gamma[\mL_Xg],\dvol = -\frac{1}{2}\int_M f\,dL(X),\dvol ,
 \end{multline*}
 where the final equality uses~\eqref{eqn:dL_Lie_derivative}.
\end{proof}

Lemma~\ref{lem:trGamma} and Lemma~\ref{lem:divGamma} immediately yield the following analogue of Schur's Lemma for CVIs.

\begin{prop}
 \label{prop:schur}
 Let $L$ be a CVI of weight $-2k$.  If $(M^n,g)$ is a connected Riemannian manifold such that $n\not=2k$ and $S=\lambda g$ for some function $\lambda\in C^\infty(M)$, then $\lambda$ and $L$ are both constant.
\end{prop}

\begin{proof}
 Set
 \begin{equation}
  \label{eqn:S0}
  S_0 := S - \frac{k}{n}Lg .
 \end{equation}
 From Lemma~\ref{lem:trGamma}, we conclude that $S_0$ is trace-free.  In particular, $S=\lambda g$ if and only if $S_0=0$.  Applying Lemma~\ref{lem:divGamma}, we compute that
 \begin{equation}
  \label{eqn:div_S0}
  \delta S_0 = \frac{n-2k}{2n}dL .
 \end{equation}
 It follows that if $n\not=2k$ and $S_0=0$, then $dL=0$.  Since $M$ is connected, $L$ is constant.  Since $n\lambda=\tr S=kL$, we see that $\lambda$ is also constant.
\end{proof}

A straightforward adaptation of the proof of De Lellis and Topping~\cite{DeLellisTopping2012} for the scalar curvature yields Theorem~\ref{thm:intro_almost_schur}.  Indeed, we prove the following stronger result:

\begin{thm}
 \label{thm:almost_schur}
 Let $L$ be a CVI of weight $-2k$.  Let $(M^n,g)$ be a compact Riemannian manifold with $n>2k$ and $\Ric\geq0$.  Then
 \begin{equation}
  \label{eqn:almost_schur}
  \int_M \left(L-\oL\right)^2\dvol \leq \frac{4n(n-1)}{(n-2k)^2}\int_M \lv S_0\rv^2\,\dvol ,
 \end{equation}
 where $\oL=(\int L)/(\int 1)$ is the average of $L$.  Moreover, if $\Ric>0$, then equality holds in~\eqref{eqn:almost_schur} if and only if $S_0=0$ for $S_0$ as in~\eqref{eqn:S0}.
\end{thm}

\begin{proof}
 Let $u\in C^\infty(M)$ be a solution of
 \begin{equation}
  \label{eqn:almost_schur_define_u}
  \Delta u = L-\oL .
 \end{equation}
 Using~\eqref{eqn:div_S0} we compute that
 \begin{multline*}
  \int_M \left(L-\oL\right)^2\dvol = -\int_M\lp dL,du\rp\,\dvol \\ = -\frac{2n}{n-2k}\int_M\lp \delta S_0,du\rp\,\dvol = \frac{2n}{n-2k}\int_M\lp S_0,\nabla^2u-\frac{1}{n}\Delta u\,g\rp\,\dvol .
 \end{multline*}
 Applying the Cauchy--Schwarz inequality, we conclude that
 \begin{equation}
  \label{eqn:almost_schur_basic_estimate}
  \int_M\left(L-\oL\right)^2\dvol \leq \frac{2n}{n-2k}\left(\int_M\lv S_0\rv^2\dvol\right)^{\frac{1}{2}}\left(\int_M \bigl(\lv\nabla^2u\rv^2 - \frac{1}{n}(\Delta u)^2\bigr)\dvol\right)^{\frac{1}{2}} .
 \end{equation}
 On the other hand, the Bochner formula and the assumption $\Ric\geq0$ together imply that
 \begin{equation}
  \label{eqn:bochner_consequence}
  \int_M\lv\nabla^2u\rv^2\dvol \leq \int_M (\Delta u)^2\dvol .
 \end{equation}
 Inserting this into~\eqref{eqn:almost_schur_basic_estimate} and using~\eqref{eqn:almost_schur_define_u} yields~\eqref{eqn:almost_schur}.  Finally, if $\Ric>0$, then equality holds in~\eqref{eqn:bochner_consequence} if and only if $u$ is constant.  It readily follows from~\eqref{eqn:almost_schur} and~\eqref{eqn:almost_schur_define_u} that if $\Ric>0$ and equality holds in~\eqref{eqn:almost_schur}, then $S_0=0$.
\end{proof}

\section{Variational Stability}
\label{sec:second_derivative}

A fundamental property of CVIs is that they are conformal gradients of Riemannian functionals: If $L$ is a CVI, then there is a \emph{conformal primitive} $\mS\colon\kC\to\bR$ with conformal gradient equal to $L$.  For the remainder of this article, when the dimension $n$ is understood from context, the \emph{standard conformal primitive} to a CVI $L$ of weight $-2k$ is given by
\[ \mS(g) = \frac{1}{n-2k}\int_M L_g\,\dvol_g \]
if $n\not=2k$, and by~\eqref{eqn:associated_lagrangian_critical_dimension} if $n=2k$.  By restricting a conformal primitive to a volume-normalized conformal class
\begin{equation}
 \label{eqn:kC0}
 \kC_0 := \left\{ e^{2\Upsilon}g \suchthat \int_M e^{n\Upsilon}\dvol_g = \int_M \dvol_g \right\}
\end{equation}
determined by a background Riemannian manifold $(M^n,g_0)$, one obtains a functional whose critical points are exactly those metrics for which $L$ is constant.

\begin{lem}
 \label{lem:crit_mS}
 Let $L$ be a CVI with conformal primitive $\mS$.  Then $(M^n,g)$ is a critical point of $\mS\colon\kC_0\to\bR$ if and only if $L$ is constant.
\end{lem}

\begin{proof}
 With the identification $T_g\kC\cong C^\infty(M)$, we see that
 \begin{equation}
  \label{eqn:TgkC0}
  T_g\kC_0 = \left\{ \Upsilon\in C^\infty(M) \suchthat \int_M \Upsilon\,\dvol_g = 0 \right\} .
 \end{equation}
 The result follows readily from the method of Lagrange multipliers.
\end{proof}

As discussed in the introduction, there are many examples of CVIs for which the extreme value of the standard conformal primitive in $\kC_0$ has geometric content.  In this section we develop tools aimed at studying all CVIs for which this is the case.  These tools are based on the behavior of $\mS\colon\kC_0\to\bR$ to second order near its critical points.

Given a CVI $L$ with standard conformal primitive $\mS$ and a Riemannian manifold $(M^n,g)$, one cannot expect that the minimizers of $\mS\colon\kC_0\to\bR$, if they exist, are unique or even isolated.  This is already observed on the round spheres, due to the existence of nontrivial conformal Killing fields.  A \emph{conformal Killing field} is a vector field $X$ on $M$ such that $\mL_Xg=2\Upsilon g$ for some $\Upsilon\in C^\infty(M)$; if $\Upsilon$ does not vanish identically, we say that $X$ is \emph{nontrivial}.  Note that the trivial conformal Killing fields are exactly the Killing vector fields, and that the condition that $X$ be a conformal Killing field is conformally invariant.

Let $\mK$ denote the space of conformal Killing fields.  If $X\in\mK$, then its flow $\psi_t$ generates a curve $t\mapsto\psi_t^\ast g$ in $\kC_0$.  Since $\mS$ is a Riemannian functional, it is constant along this curve.  It follows that the functional $\mS$ is constant on the submanifold
\begin{equation}
 \label{eqn:kK}
 \kK(g) := \left\{ \psi^\ast g \suchthat \psi\in\Con_0(M) \right\} \subset \kC_0 .
\end{equation}
where $\Con_0(M)\subset\Diff(M)$ is the connected component of the identity in the Lie group $\Con(M)$ of diffeomorphisms which preserve $\kC$,
\[ \Con(M) := \left\{ \psi\in\Diff(M) \suchthat \psi^\ast g \in [g] \right\} . \]
When $g$ is understood by context, we write $\kK$ for $\kK(g)$.  Note that the tangent space to $\kK$ at $g\in\kC_0$ is, via the identificiation $T_g\kC\cong C^\infty(M)$,
\begin{equation}
 \label{eqn:TgkK}
 T_g\kK = \left\{ \Upsilon\in C^\infty(M) \suchthat \text{$\mL_Xg=2\Upsilon g$ for some $X\in\mK$} \right\} .
\end{equation}

\begin{remark}
 \label{rk:kazdan_warner}
 As observed by Gover and {\O}rsted~\cite{GoverOrsted2010}, the invariance of $\mS$ along $\kK$ yields a Kazdan--Warner-type obstruction to solutions of $L_g=f$ in $\kC_0$; see also~\cite{Bourguignon1985,DelanoeRobert2007}.
\end{remark}

It is a well-known fact that if $(M^n,g)$ is a Riemannian manifold with constant scalar curvature which admits a nontrivial conformal Killing field, then $2J$ is an eigenvalue of the Laplacian~\cite{Lichnerowicz1958,Yano1957}, where $J=R/2(n-1)$ is the trace of the Schouten tensor.  Indeed, for every $\Upsilon\in T_g\kK$, it holds that
\[ -\Delta\Upsilon = 2J\Upsilon . \]
This fact generalizes to CVIs:

\begin{prop}
 \label{prop:S1_eigenvalue}
 Let $L$ be a CVI of weight $-2k$.  Let $(M^n,g)$ be a Riemannian manifold such that $L$ is constant.  For all $\Upsilon\in T_g\kK$, it holds that
 \[ -\delta\left(T(d\Upsilon)\right) = 2kL\Upsilon . \]
\end{prop}

\begin{proof}
 Let $X$ be a conformal Killing field with $\mL_Xg=2\Upsilon g$.  Since $L$ is constant, $XL=0$.  Lemma~\ref{lem:dL_Lie_derivative} and~\eqref{eqn:cvi_linearization} then imply that
 \[ 0 = DL(\Upsilon) = -2kL\Upsilon - \delta\left(T(d\Upsilon)\right) . \qedhere \]
\end{proof}

The upshot of the above discussion is that if $(M^n,g)$ is a minimizer of $\mS\colon\kC_0\to\bR$, then $L$ is constant and $\mS$ is minimized along $\kK$.  We are interested in those critical points for which $\mS$ is minimized to second order along $\kK$.

\begin{defn}
 \label{defn:variational_stability}
 A CVI $L$ is \emph{variationally stable} at a critical point $(M^n,g)$ of the standard conformal primitive $\mS\colon\kC_0\to\bR$ if the Hessian $D^2\mS\colon T_g\kC_0\to\bR$ is nonnegative and such that $\ker D^2\mS = T_g\kK$.
\end{defn}

Recall that if $(M^n,g)$ is a critical point of $\mS\colon\kC_0\to\bR$, then the Hessian $D^2\mS$ of $\mS$ is defined as a quadratic form on $T_g\kC_0$ by
\[ D^2\mS(\Upsilon) = \left.\frac{d^2}{dt^2}\right|_{t=0} \mS(g_t), \]
where $g_t$ is a smooth one-parameter family of metrics in $\kC_0$ with $g_0=g$ and $\left.\frac{\partial}{\partial t}\right|_{t=0}g_t=2\Upsilon g$.  Since $g$ is a critical point of $\mS$, this is independent of the choice of curve $g_t$.

One expects that for a given CVI $L$, the second conformal variation $D^2\mS$ is independent of the choice of conformal primitive.  The following lemma verifies this by computing $D^2\mS$ in terms of the conformal linearization of $L$.

\begin{lem}
 \label{lem:second_variation}
 Let $\mS$ be a conformal primitive for a CVI $L$ of weight $-2k$.  Let $(M^n,g)$ be a compact Riemannian manifold such that $L$ is constant.  Then the Hessian of $\mS\colon\kC_0\to\bR$ at $g$ is
 \begin{equation}
  \label{eqn:second_variation}
  D^2\mS(\Upsilon) = \int_M \left[ \lp T(d\Upsilon),d\Upsilon\rp - 2kL\Upsilon^2 \right]\,\dvol .
 \end{equation}
\end{lem}

\begin{proof}
 Let $t\mapsto g_t\in\kC_0$ be a curve with $g_0=g$ and $\left.\frac{\partial}{\partial t}\right|_{t=0}g_t=2\Upsilon g$.  Denote by $\Upsilon_t\in C^\infty(M)$ the function satisfying $\frac{\partial}{\partial t}g_t=2\Upsilon_tg_t$ and denote $\Upsilon_t^\prime=\frac{\partial}{\partial t}\Upsilon_t$.  Thus $\Upsilon=\Upsilon_0$.  Denote $\Upsilon^\prime=\Upsilon_0^\prime$.  A direct computation using~\eqref{eqn:cvi_linearization} yields
 \begin{align*}
  \left.\frac{d^2}{dt^2}\right|_{t=0}\mS(g_t) & = \left.\frac{d}{dt}\right|_{t=0}\int_M \Upsilon_tL_{g_t}\,\dvol_{g_t} \\
  & = \int_M \left[ \lp T(d\Upsilon),d\Upsilon\rp + (n-2k)L\Upsilon^2 + L\Upsilon^\prime \right]\,\dvol .
 \end{align*}
 On the other hand, since $g_t\in\kC_0$, it holds that
 \[ 0 = \left.\frac{d^2}{dt^2}\right|_{t=0}\int_M\dvol_{g_t} = n\int_M\left[ \Upsilon^\prime + n\Upsilon^2\right]\,\dvol . \]
 Since $L$ is constant, these two identities together yield~\eqref{eqn:second_variation}.
\end{proof}

Lemma~\ref{lem:second_variation} implies that variational stability is equivalent to a spectral condition on the conformal linearization.

\begin{cor}
 \label{cor:second_variation}
 Let $\mS$ be a conformal primitive for a CVI $L$.  Let $(M^n,g)$ be a compact Riemannian manifold with $L$ constant.  Then the Hessian of $\mS\colon\kC_0\to\bR$ at $g$ is
 \begin{equation}
  \label{eqn:second_variation_DL}
  D^2\mS(\Upsilon) = \int_M \Upsilon\,DL(\Upsilon)\,\dvol .
 \end{equation}
 In particular, $L$ is variationally stable at $g$ if and only if $DL$ is a nonnegative operator and $\ker DL=T_g\kK$.
\end{cor}

\begin{proof}
 It follows immediately from~\eqref{eqn:cvi_linearization} that
 \[ \int_M \Upsilon\,DL(\Upsilon) = \int_M \left[ \lp T(d\Upsilon),d\Upsilon\rp - 2kL\Upsilon^2 \right]\dvol . \]
 Lemma~\ref{lem:second_variation} thus yields~\eqref{eqn:second_variation_DL}.  The final conclusion follows immediately from the definition of variational stability and Proposition~\ref{prop:S1_eigenvalue}.
\end{proof}

One of our long-term goals is to find necessary and sufficient conditions for a CVI to be extremized at Einstein metrics of positive scalar curvature in their conformal class.  Two obvious necessary conditions are that the CVI be constant at all Einstein metrics and be variationally stable at all Einstein metrics with positive scalar curvature.  For many explicit examples of CVIs, it is clear whether they are constant at all Einstein metrics (cf.\ Section~\ref{sec:weight}).  Checking whether such CVIs are variationally stable at Einstein metrics with positive scalar curvature is made relatively simple by means of the Lichnerowicz--Obata Theorem~\cite{Lichnerowicz1958,Obata1962}.

\begin{lem}
 \label{lem:lichnerowicz_obata}
 Let $(M^n,g)$ be an Einstein manifold with positive scalar curvature.  Then $\lambda_1(-\Delta)\geq 2J$ with equality if and only if $(M^n,g)$ is isometric to the round $n$-sphere.  In particular, if $\Upsilon\in T_g\kK$, then $\Upsilon=0$ or $(M^n,g)$ is isometric to the round $n$-sphere.
\end{lem}

\begin{proof}
 Since $\Ric=(2J/n)g>0$, the Lichnerowicz--Obata Theorem implies that $\lambda_1(-\Delta)\geq 2J$ with equality if and only if $(M^n,g)$ is isometric to the round $n$-sphere~\cite{Lichnerowicz1958,Obata1962}.

 Suppose that $\Upsilon\in T_g\kK$.  Let $X\in\mK$ be such that $\mL_Xg=2\Upsilon g$.  Since $J$ is constant, Proposition~\ref{prop:S1_eigenvalue} implies that $-\Delta\Upsilon=2J\Upsilon$.  By the previous paragraph, if $\Upsilon\not=0$, then $(M^n,g)$ is isometric to the round $n$-sphere.
\end{proof}

Note that Lemma~\ref{lem:lichnerowicz_obata} recovers the well-known characterization of $\kK(g)$ at an Einstein metric with positive scalar curvature (see~\cite{YanoNagano1959}).  In particular, it implies that if $L$ is a CVI which is variationally stable at all Einstein metrics with positive scalar curvature, then for each Einstein manifold $(M^n,g)$, either $g$ minimizes $\mS$ to second order or $(M^n,g)$ is isometric to the round $n$-sphere and $\kK(g)$ minimizes $\mS$ to second order.

\section{Ellipticity and solvability of $L$}
\label{sec:singular}

We now turn to solving, at least locally, the problem of prescribing a CVI within a given conformal class.  One expects a general existence theory when restricting to the class of elliptic CVIs.

\begin{defn}
 A CVI $L$ is \emph{elliptic at $(M^n,g)$} if its conformal linearization at $g$ is elliptic under the identification $T_g\kC\cong C^\infty(M)$.  A CVI $L$ is \emph{elliptic} if it is elliptic at every Riemannian manifold.
\end{defn}

Recalling the formula~\eqref{eqn:cvi_linearization} for the conformal linearization of a CVI, we see that $L$ is elliptic at $(M^n,g)$ if and only if the operator $-\delta Td$ is elliptic.  For example, the scalar curvature and Branson's $Q$-curvatures~\cite{Branson1995} are elliptic CVIs, but the $\sigma_2$-curvature is only elliptic in a subcone of $\kC$ (see~\cite{Viaclovsky2000}).  Further examples are discussed in Section~\ref{sec:weight}.  The existence of elliptic cones for a CVI is a general phenomenon:

\begin{lem}
 \label{lem:elliptic_cones}
 Let $L$ be a CVI.  Given a conformal class $\kC$, set
 \[ \mC_L = \left\{ g\in\kC \suchthat \text{$L$ is elliptic at $g$} \right\} . \]
 Then $\mC_L$ is an open cone.
\end{lem}

\begin{proof}
 It follows immediately from the homogeneity of $L$ that $DL_{c^2g}=c^{-2k}DL_g$, where $-2k$ is the weight of $L$.  Hence $\mC_L$ is a cone.

 Let $T$ be as in~\eqref{eqn:cvi_linearization}.  Since $L$ is natural, we readily conclude that $T$ is a natural differential operator.  Hence $T$, and therefore $\delta Td$, is continuous in $\kC$.  Thus $\mC_L$ is open.
\end{proof}

One consequence of ellipticity is that variational stability implies local stability.

\begin{prop}
 \label{prop:local_stability}
 Let $L$ be a CVI which is variationally stable and elliptic at $(M^n,g)$ and let $\mS$ be a conformal primitive for $L$.  If $L_g$ is constant, then there is a neighborhood $U\subset\kC_0$ of $g$ such that
 \[ \mS(\hg) \geq \mS(g) \]
 for all $\hg\in U$, with equality if and only if $\hg\in U\cap\kK$.
\end{prop}

\begin{proof}
 This is a consequence of the Morse Lemma for Banach spaces in normal variables~\cite[p.\ 543]{FischerMarsden1975}:

 Let $\ell\in\bN$ be sufficiently large so that the space $\kC_0^{\ell,\alpha}$ of volume-normalized $C^{\ell,\alpha}$-metrics conformal to $g$ is a Banach manifold and $\mS\colon\kC^{\ell,\alpha}\to\bR$ is $C^2$.  Then $\kK$ is a finite-dimensional submanifold of $\kC^{\ell,\alpha}$ and $\mS\rv_{\kK}$ is constant.  Since $L_g$ is constant, the derivative $d\mS\colon T_g\kC_0^{\ell,\alpha}\to\bR$ vanishes identically.  Since $\kK$ is the orbit of $g$ under the action of $\Con_0(M)$, we conclude that $d\mS$ vanishes along $\kK$.

 Let $E_g\subset T_g\kC^{\ell,\alpha}$ be the orthogonal complement to $T_g\kK$ with respect to the inner product $\llp\cdot,\cdot\rrp$.  By assumption, $L_g$ is the gradient of $\mS$.  Since $L$ is a CVI, $DL_g$ is formally self-adjoint, and hence maps $E_g$ to $E_g$.  Since $L$ is variationally stable at $g$, it holds that $DL_g\colon E_g\to E_g$ is positive definite.  Thus $DL_g$ is injective.  Finally, since $L$ is an elliptic CVI at $g$, its gradient $DL_g$ is Fredholm and of index zero.  Thus $DL_g$ is surjective.  We may thus apply~\cite[Lemma~5, p.\ 543]{FischerMarsden1975} to conclude that there is a neighborhood $U\subset\kC_0^{\ell,\alpha}$ of $\kK$ such that $\mS(\hg)\geq\mS(g)$ for all $\hg\in U$, with equality if and only if $\hg\in U\cap\kK$, as desired.
\end{proof}

Another application of ellipticity is to the solvability of the equation $L_g=f$.  However, even for elliptic CVIs, there are obstructions to this problem (e.g.\ Remark~\ref{rk:kazdan_warner}).  Infinitesimally, the failure of $\Gamma$ to be surjective, or equivalently $\Gamma^\ast$ to be injective, gives an obstruction to this problem.

\begin{lem}
 \label{lem:injective_symbol}
 Let $L$ be a CVI which is elliptic at $(M^n,g)$.  Then $\Gamma^\ast$ has injective principal symbol.  In particular,
 \begin{equation}
  \label{eqn:splitting_injective_symbol}
  C^\infty(M) = \ker\Gamma^\ast \oplus \im \Gamma .
 \end{equation}
\end{lem}

\begin{proof}
 By Lemma~\ref{lem:trGamma}, the trace of the principal symbol of $\Gamma^\ast$ is the principal symbol of $DL\colon T_g\kC\to C^\infty(M)$.  Since the latter is invertible, $\Gamma^\ast$ has injective principal symbol.  The splitting~\eqref{eqn:splitting_injective_symbol} then follows from the Berger--Ebin splitting theorem~\cite{BergerEbin1969}.
\end{proof}

Lemma~\ref{lem:injective_symbol} provides the essential tool to prove Theorem~\ref{thm:intro_corollary}.

\begin{proof}[Proof of Theorem~\ref{thm:intro_corollary}]
 Since $\ker\Gamma^\ast=\{0\}$, Lemma~\ref{lem:injective_symbol} implies that $L\colon\kM\to C^\infty(M)$ is a submersion at $g$.  The conclusion then follows from the Implicit Function Theorem on Banach spaces (see~\cite{Lang1995}).
\end{proof}

Lemma~\ref{lem:injective_symbol} and Theorem~\ref{thm:intro_corollary} show that $\ker\Gamma^\ast$ is an object of geometric interest.  Motivated by other natural geometric problems involving CVIs (e.g.\ Section~\ref{sec:critical} below), we introduce the following terminology.

\begin{defn}
 Let $L$ be a CVI of weight $-2k$.  A Riemannian manifold $(M^n,g)$ is \emph{$L$-critical} if there is a nontrivial function $f\in C^\infty(M)$ and a constant $\lambda\in\bR$ such that $\Gamma^\ast (f) = \lambda g$.  A Riemannian manifold $(M^n,g)$ is \emph{$L$-singular} if $\ker\Gamma^\ast\not=\{0\}$.
\end{defn}

By~\eqref{eqn:first_derivative} and Proposition~\ref{prop:schur}, if $n>2k$ and $(M^n,g_0)$ is a critical point of the restriction of the total $L$-curvature functional to
\[ \kM_0 := \left\{ g\in\kM \suchthat \Vol_g(M) = \Vol_{g_0}(M) \right\}, \]
then $(M^n,g_0)$ is $L$-critical.

Theorem~\ref{thm:intro_corollary} states that if $(M^n,g)$ is not $L$-critical, then the equation $L_{\og}=\psi$ can be solved locally in $\kM$.  Under the additional assumption that $L_g\equiv0$, this problem is globally unobstructed (cf.\ \cite{LinYuan2015a}).

\begin{cor}
 \label{cor:deformation_from_zero}
 Let $L$ be a nontrivial CVI which is elliptic at $(M^n,g)$.  If $L_g=0$ and $(M^n,g)$ is not $L$-singular, then for every bounded $\psi\in C^\infty(M)$, there is a metric $\og\in\kM$ such that $L_{\og}=\psi$.
\end{cor}

\begin{proof}
 Let $-2k$, $k\in\bN$ be the weight of $L$.  Let $V$ be as in Theorem~\ref{thm:intro_corollary} and choose $\varepsilon>0$ be sufficiently small so that $\varepsilon^{2k}\psi\in V$.  Then there is a $\hg\in\kM$ such that $L_{\hg}=\varepsilon^{2k}\psi$.  Then $\og:=\varepsilon^2\hg$ is such that $L_{\og}=\psi$.
\end{proof}

In order to apply Theorem~\ref{thm:intro_corollary} and Corollary~\ref{cor:deformation_from_zero}, we would like a test for whether a Riemannian manifold is not $L$-singular.  For elliptic CVIs of low weight, Theorem~\ref{thm:intro_Lconst} gives a sufficient condition: If $L_g$ is nonconstant, then $(M,g)$ is not $L$-singular.  Indeed, we conclude that $(M,g)$ is not $L$-critical.  Generalizing the approach to $R$-critical~\cite{MiaoTam2008} and $Q_4$-critical metrics~\cite{LinYuan2015b}, we prove Theorem~\ref{thm:intro_Lconst} by using the weak unique continuation property.

\begin{defn}
 A differential operator $F\colon C^\infty(M)\to C^\infty(M)$ on $M$ satisfies the \emph{weak unique continuation property} if the only solution $\phi\in C^\infty(M)$ of $F(\phi)=0$ such that $\phi\rv_U\equiv0$ for some open set $U\subset M$ is the constant function $\phi\equiv0$.
\end{defn}

\begin{thm}
 \label{thm:unique_continuation}
 Let $L$ be a CVI.  Suppose that $(M^n,g)$ is such that $DL_g$ satisfies the weak unique continuation property.  Let $f\in C^\infty(M)$ be such that $\Gamma^\ast(f)=\lambda g$ for some constant $\lambda\in\bR$.  If $f\not\equiv0$, then $L_g$ is constant.
\end{thm}

\begin{proof}
 Since $\Gamma^\ast(f)=\lambda g$, Lemma~\ref{lem:divGamma} implies that
 \begin{equation}
  \label{eqn:LsLs1}
  f\,dL=0 .
 \end{equation}
 Suppose that $L_g$ is not constant.  Then there is a nonempty open set $U\subset M$ such that $(dL)_{x}\not=0$ for all $x\in U$.  It follows from~\eqref{eqn:LsLs1} that $f\rv_U\equiv0$.  Since $DL_g$ satisfies the weak unique continuation property, we conclude that $f\equiv0$.
\end{proof}

There are some ways to conclude that $L_g$ is constant without knowing that $DL_g$ satisfies the weak unique continuation property.  For example, $L$-critical manifolds which are not $L$-singular have constant $L$-curvature.  A more refined statement is as follows:

\begin{prop}
 \label{prop:Lcrit-nonLsing}
 Let $L$ be a CVI.  Let $(M^n,g)$ be a Riemannian manifold such that there is a function $f\in C^\infty(M)$ such that $\Gamma^\ast(f)=\lambda g$ for some nonzero constant $\lambda\in\bR$.  Then $L_g$ is constant.
\end{prop}

\begin{proof}
 Suppose to the contrary that $L_g$ is nonconstant.  As in the proof of Theorem~\ref{thm:unique_continuation}, there is a nonempty open set $U\subset M$ such that $f\rv_U\equiv0$.  On the other hand, Lemma~\ref{lem:trGamma} and the condition $\Gamma^\ast(f)=\lambda g$ together imply that
 \begin{equation}
  \label{eqn:LsLs2}
  \delta\left(T(df)\right) = -2kLf - 2n\lambda .
 \end{equation}
 By considering~\eqref{eqn:LsLs2} in $U$, we see that $\lambda=0$, a contradiction.
\end{proof}

We are aware of conditions which imply the weak unique continuation property for two general classes of differential operators.  The first~\cite{Protter1960} applies to differential operators of the form $\Delta^k+\mathrm{l.o.t.}$, where ``$\mathrm{l.o.t.}$'' refers to terms of order at most $\lfloor 3k/2\rfloor$.  The second~\cite{Zuily1983} applies to differential operators for which the roots of the principal symbol have low multiplicity.  The latter situation is most relevant for the study of elliptic CVIs, and can be summarized as follows:

\begin{thm}
 \label{thm:zuily}
 Let $F\colon C^\infty(M)\to C^\infty(M)$ be a natural elliptic differential operator on $(M^n,g)$.  Fix a nonzero vector $v\in\bR^n$.  Given a vector $\xi\in\bR^n$ orthogonal to $v$, denote $s_\xi(t):=\sigma(F)(\xi+tv)$, where $\sigma(F)$ is the principal symbol of $F$.  Let
 \[ N(x) := \#\left\{ r\in\bC \suchthat s_\xi(r)=0, \Imaginary r>0 \right\} \]
 denote the number of complex roots of $s_\xi$ with positive imaginary part is constant over $M$.  Suppose either
 \begin{enumerate}
  \item $N(x)\leq 2$ for all $x\in M$, or
  \item $N(x)=3$ for all $x\in M$.
 \end{enumerate}
 Then $F$ satisfies the weak unique continuation property.
\end{thm}

\begin{proof}
 Since $F$ is elliptic, the roots of $s_\xi$ are non-real and come in pairs $a\pm ir$, $r>0$.  Since $F$ is natural, $\sigma(F)$ is invariant under the action of the orthogonal group, and hence the multipicities of these roots depends only on the angle between $v$ and $\xi$.  The conclusion then follows from~\cite[Theorem~1.1]{Zuily1983} if $N\leq2$ and from~\cite[Theorem~3.1]{Zuily1983} (see also~\cite{Watanabe1979}) if $N=3$.
\end{proof}

Theorem~\ref{thm:unique_continuation} and Theorem~\ref{thm:zuily} together yield a proof of Theorem~\ref{thm:intro_Lconst}.

\begin{proof}[Proof of Theorem~\ref{thm:intro_Lconst}]
 Suppose first that $DL_g$, as a differential operator, is of order at most four.  Then Theorem~\ref{thm:zuily} implies that $DL_g$ satisfies the weak unique continuation property.  From Theorem~\ref{thm:unique_continuation} we conclude that if $L_g$ is nonconstant, then $(M^n,g)$ is not $L$-critical.

 Suppose next that $L_g$ has weight $-2k$, $k\in\{1,2,3\}$.  From Weyl's invariant theory, we conclude that $DL_g$ is of order at most $2k$.  Moreover, if $DL_g$ is of order exactly $2k$, then the leading order term of $DL_g$ is a nonzero multiple of $(-\Delta)^k$.  In particular, either $DL_g$ is of order at most four or the principal symbol of $DL_g$ is a nonconstant multiple of the principal symbol of $(-\Delta)^3$.  In either case, Theorem~\ref{thm:unique_continuation} and Theorem~\ref{thm:zuily} together imply that if $L_g$ is nonconstant, then $(M^n,g)$ is not $L$-critical.
\end{proof}

An immediate consequence of Theorem~\ref{thm:intro_Lconst} is that if $L$ is an elliptic CVI of weight $-2k$, $k\leq3$, at $(M^n,g)$, $n\not=2k$, and $(M^n,g)$ is $L$-critical, then $g$ is a critical point of the total $L$-curvature functional in the volume-normalized conformal class $\kC_0$.  One wonders if such manifolds are in fact critical points of the total $L$-curvature functional in the space of volume-normalized Riemannian metrics $\kM_0$.  A partial affirmative answer is given by the following stronger form of Theorem~\ref{thm:intro_nontrivial-Lsingular}:

\begin{thm}
 \label{thm:nontrivial-L-singular}
 Let $L$ be a CVI of weight $-2k$ and let $(M^n,g)$ be an $L$-critical manifold for which $L_g$ is constant.  Let $f\in C^\infty(M)$ be a nontrivial function such that $\Gamma^\ast(f)=\lambda g$ for some $\lambda\in\bR$.  If $L$ is elliptic and variationally stable at $g$, then $f=f_0+\of$, where $\of=(\int f)/(\int 1)$ is the average of $f$ and $f_0\in\kK$.  In particular, if $(M^n,g)$ has no nontrivial conformal Killing fields and $n\not=2k$, then $(M^n,g)$ is a critical point of the total $L$-curvature functional $\mS\colon\kM_0\to\bR$.
\end{thm}

\begin{remark}
 By Theorem~\ref{thm:intro_Lconst}, if $k\leq3$, then we do not need to assume that $L_g$ is constant.
\end{remark}

\begin{proof}
 Since $\Gamma^\ast(f)=\lambda g$, we conclude from~\eqref{eqn:LsLs2} that
 \[ \int_M \left[\lp T(df),df\rp - 2kLf^2\right]\,\dvol = 2n\lambda\int_M f\,\dvol . \]
 Since $L$ is constant, \eqref{eqn:LsLs2} implies that $kL\of=-n\lambda$.  Set $f_0=f-\of$.  Then
 \[ \int_M \left[\lp T(df_0),df_0\rp - 2kLf_0^2\right]\,\dvol = 0 . \]
 Since $L$ is variationally stable, it follows that $f_0\in\kK$.  If $f_0=0$, then $f$ is a (necessarily nonzero) constant.  Hence $S=-(\lambda/f) g$.  We then conclude from~\eqref{eqn:first_derivative} that $(M^n,g)$ is a critical point of $\mS\colon\kM_0\to\bR$.
\end{proof}

As a consequence of Theorem~\ref{thm:intro_Lconst}, we can generalize the characterizations of nontrivial Einstein $R$-critical manifolds~\cite{MiaoTam2011} and nontrivial Einstein $Q_4$-critical manifolds~\cite{LinYuan2015a} to the case of CVIs which are variationally stable and elliptic at Einstein metrics with positive scalar curvature.

\begin{cor}
 \label{cor:characterize_sphere}
 Let $(M^n,g)$ be an Einstein manifold with positive scalar curvature and let $L$ be a CVI which is constant, variationally stable, and elliptic at $g$.  If there is a nonconstant function $f$ such that $\Gamma^\ast(f)=\lambda g$ for some $\lambda\in\bR$, then $(M^n,g)$ is homothetic to the round $n$-sphere and $f=ax+b$ for constants $a,b\in\bR$, where $x$ is a first spherical harmonic.
\end{cor}

\begin{remark}
 If $L$ is of weight $-2k\geq-6$, then we do not need to assume that $L_g$ is constant.
\end{remark}

\begin{proof}
 This follows immediately from Lemma~\ref{lem:lichnerowicz_obata} and Theorem~\ref{thm:nontrivial-L-singular}.
\end{proof}

\section{A variational interpretation of $L$-critical metrics}
\label{sec:critical}

In the case when $L$ is the scalar curvature, Miao and Tam~\cite{MiaoTam2008} gave a variational characterization of $L$-critical metrics as critical points of the volume functional on the space of metrics of fixed constant scalar curvature on a compact manifold with fixed boundary metric.  This characterization plays an important role in studying scalar curvature rigidity on manifolds with boundary (e.g.\ \cite{CorvinoEichmairMiao2013,MiaoTam2012,Yuan2016}).  In this section we give an analogous characterization of $L$-critical metrics near a metric $g$ at which $L$ is elliptic and variationally stable, in the case when the underlying manifold is closed (i.e.\ has no boundary).

Roughly speaking, the above claim is that solutions to $\Gamma^\ast(f)=\lambda g$ correspond to critical points of the volume functional on the space of metrics with $L$-constant; see Theorem~\ref{thm:L-critical} below for a precise statement.  There are two approaches to this problem, based on two ways to eliminate the scaling symmetries of the equation $\Gamma^\ast(f)=\lambda g$.  One such symmetry is that if $\Gamma^\ast(f)=\lambda g$, then $\Gamma^\ast(cf)=c\lambda g$ for all $c\in\bR$.  The other is that if $\Gamma_g^\ast(f)=\lambda g$, then $\Gamma_{c^2g}^\ast(f)=c^{-2k}\lambda(c^2g)$, where $-2k$ is the weight of $L$.  Thus we can eliminate the scaling symmetry either by fixing a normalization of $f$ and considering the volume functional on
\[ \mM_L = \left\{ g\in\kM \suchthat \Delta_gL_g = 0 \right\} \]
or instead by fixing the normalization of $L$ by studying the volume functional on
\begin{equation}
 \label{eqn:mMLK}
 \mM_L^K = \left\{ g \in \kM \suchthat L_g=K \right\}
\end{equation}
for some $K\in\bR$ fixed.  We shall take the latter approach, following~\cite{CorvinoEichmairMiao2013}.  The former approach was used by Koiso~\cite{Koiso1979} to study the manifold of metrics with constant scalar curvature; see also~\cite[Section~4.F]{Besse}.

Assume that $\mM_L^K$ is not empty.  We begin by giving sufficient conditions on $L$ and $g\in\mM_L^K$ to guarantee that there is a neighborhood of $g$ in $\mM_L^K$ which has the structure of a Banach manifold when completed with respect to the $C^{\ell,\alpha}$-norm for $\ell$ sufficiently large.  A key point is that these conditions depend only on the conformal linearization of $L$ at $g$.

\begin{lem}
 \label{lem:mMLK_manifold}
 Let $L$ be a CVI, fix $K\in\bR$, and let $\mM_L^K$ be as in~\eqref{eqn:mMLK}.  If $g\in\mM_L^K$ is such that $L$ is elliptic at $g$ and $(M,g)$ is not $L$-singular, then there is a neighborhood $U\subset\kM$ of $g$ such that for every $\og\in U$, there is a function $u\in C^\infty(M)$ such that $e^{2u}\og\in\mM_L^K$.  Moreover, with respect to the identification $T_g\kM\cong S^2T^\ast M$,
 \begin{equation}
  \label{eqn:TgmMLK}
  T_g\mM_L^K \cong \left\{ h\in S^2T^\ast M \suchthat DL_g[h] = 0 \right\} .
 \end{equation}
\end{lem}

\begin{proof}
 Consider the map $\mF\colon\kM\times C^\infty(M)\to C^\infty(M)$ given by $\mF(\og,u)=L(e^{2u}\og)$.  With $D_2$ denoting differentiation in the $C^\infty(M)$ factor, we see that $D_2\mF(g,0)(\Upsilon)=DL_g(\Upsilon)$.  Since $L$ is elliptic and $\ker\Gamma^\ast=\{0\}$, Lemma~\ref{lem:injective_symbol} implies that $D_2\mF(g,0)$ is surjective.  The Implicit Function Theorem then implies that there is a neighborhood $U\subset\kM$ of $g$ and a smooth function $\Phi\colon U\to C^\infty(M)$ such that $\exp\left(2\Phi(\og)\right)\og\in\mM_L^K$ for every $\og\in U$.  The identification~\eqref{eqn:TgmMLK} readily follows.
\end{proof}

An adaption of the argument of Corvino, Eichmair and Miao~\cite{CorvinoEichmairMiao2013} leads to the following variational characterization of solutions to $\Gamma^\ast(f)=g$.

\begin{thm}
 \label{thm:L-critical}
 Let $M$ be a compact manifold and let $K\in\bR$.  Let $L$ be a CVI and assume the space $\mM_L^K$ defined by~\eqref{eqn:mMLK} is nonempty.  If $g$ is such that $L$ is elliptic at $g$ and $\ker\Gamma^\ast=\{0\}$, then $g$ is a critical point of $\Vol\colon\mM_L^K\to\bR$ if and only if there is an $f\in C^\infty(M)$ such that $\Gamma^\ast(f)=g$.
\end{thm}

\begin{proof}
 Suppose that $g$ is a critical point of $\Vol\colon\mM_{L}^K\to\bR$.  Let $h\in T_g\kM$ and set $g_t=g+th$ for $t\in(-\varepsilon,\varepsilon)$ with $\varepsilon>0$ sufficiently small.  By Lemma~\ref{lem:mMLK_manifold}, there is a smooth family $\Upsilon_t\in C^\infty(M)$ such that $\Upsilon_0=0$ and $L_{\og_t}=K$ for $\og_t:=e^{2\Upsilon_t}g_t\in\mM_{L}^K$.  By construction of $\og_t$, we have that
 \begin{equation}
  \label{eqn:L-critical_obs1}
  0 = \left.\frac{\partial}{\partial t}\right|_{t=0}L_{\og_t} = DL[h+2\Upsilon^\prime g] = DL[h] + DL(\Upsilon^\prime),
 \end{equation}
 where $\Upsilon^\prime:=\left.\frac{\partial}{\partial t}\right|_{t=0}\Upsilon_t$.  Since $g$ is a critical point of $\Vol$ on $\mM_L^K$, we also have that
 \begin{equation}
  \label{eqn:L-critical_obs2}
  0 = 2\left.\frac{\partial}{\partial t}\right|_{t=0}\Vol(\og_t) = \int_M \left[ \tr_g h + 2n\Upsilon^\prime\right]\,\dvol_g .
 \end{equation}
 Since $\ker\Gamma^\ast=\{0\}$ and $L$ is elliptic at $g$, Lemma~\ref{lem:injective_symbol} implies that there is a unique function $f$ such that $DL(f)=2n$.  Therefore
 \[ 2n\int_M \Upsilon^\prime\,\dvol = \int_M \Upsilon^\prime\,DL(f)\,\dvol = \int_M f\,DL(\Upsilon^\prime)\,\dvol , \]
 where the last equality uses the fact that $L$ is conformally variational.  Using~\eqref{eqn:L-critical_obs1}, we conclude that
 \[ 2n\int_M \Upsilon^\prime\,\dvol = -\int_M f\,DL[h]\,\dvol = -\int_M \lp h, \Gamma^\ast(f)\rp\,\dvol . \]
 Equation~\eqref{eqn:L-critical_obs2} then implies that
 \[ \int_M \tr h\,\dvol = \int_M \lp h,\Gamma^\ast(f)\rp\,\dvol . \]
 Since $h$ is arbitrary, we conclude that $\Gamma^\ast(f)=g$.

 Conversely, suppose that $f$ satisfies $\Gamma^\ast(f)=g$.  Let $h\in T_g\mM_L^K$ and let $t\mapsto g_t$ be a smooth curve in $\mM_{L}^K$ with $g_0=g$ and $\left.\frac{\partial}{\partial t}\right|_{t=0}g_t=h$.  We then compute that
 \[ 2\left.\frac{\partial}{\partial t}\right|_{t=0}\Vol_{g_t}(M) = \int_M \lp h, g\rp\,\dvol = \int_M \lp h, \Gamma^\ast(f)\rp\,\dvol = \int_M f\,DL[h]\,\dvol = 0 , \]
 where the last equality follows from the characterization~\eqref{eqn:TgmMLK} of $T_g\mM_L^K$.
\end{proof}

\begin{remark}
 We expect similar results to hold for compact manifolds with boundary.  Carrying this out requires a discussion of scalar boundary invariants associated to $L$, which is beyond the scope of this article.  Suitable boundary invariants are known in certain cases, such as the $\sigma_k$-curvatures~\cite{Chen2009s} and the fourth-order $Q$-curvature~\cite{Case2015b}.
\end{remark}

Generalizing Theorem~\ref{thm:intro_corollary}, for an elliptic CVI $L$ on a compact Riemannian manifold $(M^n,g)$ with boundary, the problem of simultaneously prescribing $L$ and $\Vol(M)$ is locally (near $g$ in $\kM$) unobstructed if $(M^n,g)$ is not $L$-critical.  This was observed in the case when $L$ is the scalar curvature by Corvino, Eichmair and Miao~\cite{CorvinoEichmairMiao2013}.

\begin{thm}
 \label{thm:corvino_eichmair_miao}
 Let $L$ be a CVI which is elliptic at a compact Riemannian manifold $(M^n,g)$ without boundary.  If $(M^n,g)$ is not $L$-critical, then there are neighborhoods $U\subset\kM$ of $g$ and $V\subset C^\infty(M) \oplus \mathbb{R}$ of $(L(g), \Vol(g))$ such that for any $(\psi, v) \in V$, there is a metric $\hg\in U$ such that $L(\hg)=\psi$ and $\Vol_{\hg}(M) = v$.
\end{thm}

\begin{proof}
Consider the map $\Theta\colon\kM \to C^\infty(M) \oplus \bR$ defined by
\[ \Theta(g) = \left(L_g, \Vol_g(M)\right) . \]
The linearization $D\Theta_g\colon S^2T^\ast M \to C^\infty(M)\oplus\bR$ of $\Theta$ at $g$ is given by
\[ D\Theta_g[h] = \left(DL_g[h], D\Vol_g(M)[h]\right) . \]
It follows that its adjoint $\left(D\Theta_g\right)^\ast\colon C^\infty(M)\oplus\bR \to S^2T^\ast M$ is given by
\[ \left(D\Theta_g\right)^\ast(f,\lambda) = \Gamma^\ast(f) + \frac{\lambda}{2}g . \]
Since $L$ is elliptic at $g$, an argument analogous to the proof of Lemma~\ref{lem:injective_symbol} implies that $\left(D\Theta_g\right)^\ast$ has injective principle symbol.  Since $(M^n,g)$ is not $L$-critical, it holds that $\ker\left(D\Theta_g\right)^\ast=\{0\}$.  Therefore $C^\infty(M)=\im D\Theta_g$.  The conclusion now follows from the Implicit Function Theorem on Banach spaces.
\end{proof}

\section{Rigidity}
\label{sec:rigidity}

It is known that the scalar curvature~\cite{FischerMarsden1975} and the $Q$-curvature~\cite{LinYuan2015a} are locally rigid at flat metrics: Locally in $\kM$, flat metrics are the only metrics for which the corresponding scalar invariant is nonnegative.  Indeed, Schoen and Yau~\cite{SchoenYau1979_torus,SchoenYau1979s} and Gromov and Lawson~\cite{GromovLawson1980,GromovLawson1983} proved global rigidity of the scalar curvature, and Lin and Yuan~\cite{LinYuan2015a} have given partial results towards global rigidity for the $Q$-curvature.

In this section we study the corresponding question for CVIs: Under what condition can we determine that a CVI is locally rigid at flat metrics?  Similar to Section~\ref{sec:second_derivative}, we address this question through a spectral invariant determined by the CVI in question.

\begin{defn}
 A CVI $L$ of weight $-2k$, $k \in \mathbb{N}$, is \emph{infinitesimally rigid at flat metrics} if for every flat manifold $(M^n,g)$, it holds that
 \begin{equation}
  \label{eqn:infinitesimally_rigid}
  \int_M D^2L_g[h,h]\,\dvol_g \leq 0
 \end{equation}
 for all $h\in\ker\delta_g$, with equality if and only if
 \[ h \in \mP := \left\{ h\in S^2(M) \suchthat \nabla h = 0\right\} . \]
\end{defn}

Here $S^2(M)$ denotes the space of smooth sections of $S^2T^\ast M$.  Note that if $(M^n,g)$ is a flat manifold, then $\mP$ locally parameterizes the space of flat metrics near $g$ in $\kM$.

A key consequence of infinitesimal rigidity is the following spectral gap.

\begin{lem}
 \label{lem:flat_spectrum}
 Let $L$ be a CVI of weight $-2k$, $k\in\bN$, which is infinitesimally rigid at flat manifolds.  If $(M^n,g)$ is a closed flat manifold with $n\geq2k$ then there is a constant $C>0$ such that
 \begin{equation}
  \label{eqn:flat_spectrum}
  \int_M D^2L[h,h]\,\dvol \leq -C\int_M \lv\nabla^k h\rv^2\,\dvol
 \end{equation}
 for all $h\in\ker\delta$.
\end{lem}

\begin{proof}
 Since $(M^n,g)$ is flat and $L$ has weight $-2k$, it follows from Weyl's invariant theory and the formulae for the variations of the Levi-Civita connection and the Riemann curvature tensor that, modulo divergences, $D^2L[h,h]$ is a polynomial in the complete contractions of $\nabla^k h \otimes \nabla^kh$.  Since $(M^n,g)$ is flat, we may reorder the derivatives.  We may even do so across the tensor product, working modulo divergences.  With these freedoms, we see that if $\delta_g h=0$, then the only nonvanishing complete contractions are $\lv\nabla^kh\rv^2$ and $\lv\nabla^k \tr h\rv^2$, modulo divergences.  Hence there are constants $A, B\in\bR$ independent of $h$ such that
 \begin{equation}
  \label{eqn:D2L_generic_form}
  \int_M D^2L[h,h]\,\dvol = A\int_M \lv\nabla^kh\rv^2\,\dvol + B\int_M \lv\nabla^k \tr h\rv^2\,\dvol
 \end{equation}
 for all $h\in\ker\delta_g$.

 Define the injective map $E\colon C^\infty(M)\to S^2(M)$ by $E(\Upsilon):=\nabla^2f-\frac{1}{n}\Delta f\,g+\frac{1}{n}\Upsilon\,g$, where $f$ is the unique solution of the equation
 \begin{equation}
  \label{eqn:Delta_f}
  \begin{cases}
   \Delta f = -\frac{1}{n-1}\left(\Upsilon-\overline{\Upsilon}\right), \\
   \displaystyle \int f \, \dvol_g = 0 ,
  \end{cases}
 \end{equation}
 and $\overline{\Upsilon}$ denotes the average of $\Upsilon$ with respect to $\dvol_g$.  Since $g$ is flat, we see that $E$ takes values in $\ker\delta_g$.  Indeed, $E$ gives an orthogonal decomposition
 \begin{equation}
  \label{eqn:decomp_ker_delta}
  \ker \delta_g \cong S^2_{TT}(M) \oplus C^\infty(M)
 \end{equation}
 with respect to the $L^2$-inner product, where $S_{TT}^2(M):=\ker\delta_g\cap\ker\tr_g$ and $C^\infty(M)$ denotes the image of $E$.  To see this, let $h_f:=E(\tr_g h)$ and note that $h_{TT}:=h-h_f$ is an element of $S_{TT}^2(M)$.  The conclusion that $S_{TT}^2(M)$ and $C^\infty(M)$ are orthogonal follows from integration by parts and the flatness of $g$.

 Taking a non-parallel $h_{TT} \in S_{TT}^2(M)$ and applying the infinitesimal rigidity of $L$ at flat metrics, we get
 \[ \int_M D^2L[h_{TT},h_{TT}]\,\dvol = A\int_M \lv\nabla^kh_{TT}\rv^2\,\dvol < 0, \]
 and hence $A < 0$.  Similarly, let $\Upsilon\in C^\infty(M)$ and set $h_f:=E(\Upsilon)$.  Using integration by parts and the definition~\eqref{eqn:Delta_f} of $f$, we compute that
 \begin{align*}
  \MoveEqLeft[4] \int_M D^2L[h_f,h_f]\,\dvol \\
  &= A\int_M \left\lv \nabla^{k+2} f - \frac{1}{n}\nabla^k\Delta f \cdot g + \frac{1}{n} \nabla^k\Upsilon g \right\rv^2\,\dvol + B \int_M \lv\nabla^k \Upsilon \rv^2\,\dvol  \notag \\
  &= \left( \frac{A}{n-1} + B \right) \int_M \lv\nabla^k \Upsilon \rv^2\,\dvol,  \notag
\end{align*}
where we used the fact that, at flat metrics,
 \[ \int_M \lv \nabla^{k+2} f \rv^2 \,\dvol_g = \int_M \lv\Delta \nabla^k f \rv^2\,\dvol_g = \frac{1}{(n-1)^2}\int_M \lv\nabla^k \Upsilon \rv^2\,\dvol_g . \]
 Since $L$ is infinitesimally rigid at flat metrics, we conclude that $\frac{1}{n-1}A + B < 0$.

 Combining the above observations, we compute that
 \begin{align*}
  \int_M D^2L[h,h]\,\dvol &= A\int_M \lv\nabla^kh_{TT}\rv^2\,\dvol + \left( \frac{A}{n-1} + B  \right)\int_M \lv\nabla^k \tr h\rv^2\,\dvol \\
  &\leq - C \left( \int_M \lv\nabla^kh_{TT}\rv^2\,\dvol + \frac{1}{n-1} \int_M \lv\nabla^k \tr h\rv^2\,\dvol \right) \notag \\
  &= - C \int_M \lv\nabla^kh\rv^2\,\dvol \notag,
 \end{align*}
 where $C:= \min \left\{ -A, - A - (n-1) B \right\} > 0$.
\end{proof}

The spectral gap of Lemma~\ref{lem:flat_spectrum} enables us to generalize the proofs of the rigidity results for the scalar curvature~\cite{FischerMarsden1975} and the $Q$-curvature~\cite{LinYuan2015a} to provide a proof of Theorem~\ref{thm:intro_rigidity}.

\begin{proof}[Proof of Theorem~\ref{thm:intro_rigidity}]
 Let $-2k$ be the weight of $L$.  From Ebin's slice theorem~\cite{Ebin1970}, there are constants $N,\varepsilon_1>0$ depending only on $(M,g)$ and $k$ such that for all metrics $\og$ with $\lV\og-g\rV_{C^{2k}}<\varepsilon_1$, there exists a diffeomorphism $\varphi\colon M\to M$ such that $h:=\varphi^\ast\og-g$ is divergence-free on $M$ and $\lV h\rV_{C^0}\leq N\lV \og-g\rV_{C^0} \leq N\lV \og-g\rV_{C^{2k}}$.

 Consider the functional
 \[ \mF(\og) = \int_M L_{\og}\, \dvol_g. \]
 Note that the volume element $\dvol_g$ is fixed independently of $\og$.  Suppose that $\og$ is a metric such that $L(\og)\geq0$ and there is a diffeomorphism $\varphi\colon M\to M$ such that $h:=\varphi^\ast\og-g$ is divergence-free. Considering the decomposition $h = h^{||} + h^\perp$, where $h^{||} \in \mP$ and $h^\perp \in \mP^\perp$, we can take $g + h^{||}$ to be the flat background metric we considered and hence we may further assume that $h \in \mP^\perp$.

 Using Weyl's invariant theory, we readily check that $\mF(g)=0$ and
 \begin{equation}
  \label{eqn:mF_1st_deriv}
  \left.\frac{d}{dt}\right|_{t=0}\mF(g+th) = \int_M DL_g [h]\, \dvol_g = 0 .
 \end{equation}

 Next, from Lemma~\ref{lem:flat_spectrum} we conclude that
 \begin{equation}
  \label{eqn:mF_2nd_deriv}
  \left.\frac{d^2}{dt^2}\right|_{t=0}\mF(g+th) = \int_M D^2L_g[h,h]\dvol_g  \leq -C\int_M\lv\nabla^k h\rv^2\,\dvol_g .
 \end{equation}

 Using Weyl's invariant theory and the fact that the linearization of the Levi-Civita connection (resp.\ Riemann curvature tensor) is a polynomial in $\nabla h$ (resp.\ $\nabla^2h)$, we estimate that
 \begin{equation}
  \label{eqn:mF_3rd_deriv}
  \frac{d^3}{dt^3} \mF(g_t) \leq C_1\sum_{j = 0}^k \int_M \lv h\rv \lv\nabla^j h \rv^2\,\dvol_{g} \leq C_1 \lV h\rV_{C^0} \sum_{j = 0}^k  \int_M \lv\nabla^j h \rv^2\,\dvol_{g}
 \end{equation}
 for some constant $C_1>0$ depending only on $(M^n,g)$ and $k$, where $g_t=g+th$.  To achieve this estimate, we use the fact that the third derivative $\mF^{(3)}(g_t)$ is the integral of a polynomial in covariant derivatives (with respect to $g_t$) of $h$ and the Riemann curvature tensor of $g_t$, taken with respect to $\dvol_g$; using integration by parts, this can be written so that one copy of $h$ is undifferentiated, and the assumption $\lV \og-g\rV_{C^{2k}}<\varepsilon_1$ can be used to estimate the result in terms of $g$.

 Before we give a further estimate of the tail term above, we briefly discuss a generalization of the usual Poincar\'e inequality to the space of tensors on a manifold: Assume $(M, g)$ is a Riemannian manifold and fix $1 \leq p \leq \infty$ and $r,s\in\bN\cup\{0\}$.  There is a positive constant $\alpha_1 = \alpha_1(p,r,s,g)$ such that for any $(r, s)$-tensor $T$ we have
 \begin{equation}
  \label{eqn:gen_Poincare}
  \lV T - T^{||} \rV_{L^p} \leq \alpha_1 \lV \nabla T \rV_{L^p},
 \end{equation}
 where $T^{||} \in \mP_{(r,s)}:= \left\{ T \in (\otimes^r T^*M) \otimes (\otimes^s TM) \suchthat \nabla T = 0 \right\}$ is the projection of $T$ to the corresponding space of parallel tensors. Applying~\eqref{eqn:gen_Poincare} repeatedly, we get a higher-order version:
 \begin{equation}
  \label{eqn:gen_Poincare_hi_order}
  \lV T - T^{||} \rV_{L^p} \leq \alpha_j \lV \nabla^j T \rV_{L^p},
 \end{equation}
 for $j \in \mathbb{N}$.  The proof of~\eqref{eqn:gen_Poincare} is an easy generalization of the analogous statement for functions (see, for example, \cite{Evans2010}), and will be omitted.

 Applying~\eqref{eqn:gen_Poincare_hi_order} to~\eqref{eqn:mF_3rd_deriv}, we have the estimate
 \begin{equation}
  \label{eqn:mF_3rd_deriv_final}
  \frac{d^3}{dt^3} \mF(g_t) \leq C_2 \lV h\rV_{C^0} \int_M \lv\nabla^k h \rv^2\,\dvol_{g}
 \end{equation}

 Combining~\eqref{eqn:mF_1st_deriv}, \eqref{eqn:mF_2nd_deriv}, and~\eqref{eqn:mF_3rd_deriv_final} yields
 \begin{equation}
  \label{eqn:taylor_consequence}
  0 \leq \mF(\varphi^\ast\og) \leq -\left(C - C_2\lV h\rV_{C^0(M,g)}\right)\int_M \lv\nabla^kh \rv^2\,\dvol_g .
 \end{equation}
 Set $\varepsilon=\min\{\varepsilon_1,C/NC_2\}$ and suppose that $\og$ satisfies $L(\og)\geq0$ and $\lV\og-g\rV_{C^{2k}}<\varepsilon$.  Let $\varphi$ and $h$ be determined by Ebin's slice theorem.  It follows from~\eqref{eqn:taylor_consequence} that $\nabla^kh =0$. Integration by parts yields $\int\lv\nabla^{k-1}h\rv^2=0$.  Thus $\nabla^{k-1}h=0$.  Continuing in this manner yields $\nabla h=0$; i.e.\ $h$ is parallel and hence $h = 0$.  Therefore $\varphi^\ast\og=g$, and hence $\og$, is flat.
\end{proof}

In practice, we expect that Theorem~\ref{thm:intro_rigidity} will be most useful when applied to CVIs.  In this case, the condition that a CVI $L$ of weight $-2k$ be infinitesimally rigid at flat metrics implies that it must equal a nonzero multiple of the $Q$-curvature of order $2k$, up to the addition of lower order CVIs in the metric.  This is an immediate consequence of the following lemma.

\begin{lem}
 \label{lem:rigid_implies_linear_term}
 Let $L$ be a CVI of negative weight $-2k$, $k\in\bN$, which is infinitesimally rigid at flat metrics.  Then there is a constant $c\not=0$ such that for any flat manifold $(M^n,g)$, $n\geq2k$, it holds that $S_g=0$ and
 \begin{equation}
  \label{eqn:rigid_implies_linear_term}
  DL_g[h] = c(-\Delta)^k\tr_g h
 \end{equation}
 for all $h\in\ker\delta_g$.
\end{lem}

\begin{remark}
 \label{rk:rigid_requires_cvi}
 The assumption that $L$ is a CVI is necessary.  Indeed, $-R^2$ is a natural Riemannian scalar invariant which is homogeneous and satisfies~\eqref{eqn:infinitesimally_rigid} at flat metrics.  However, $-R^2$ is not conformally variational, and clearly $D(-R^2)_g\equiv0$ at flat metrics.
\end{remark}

\begin{proof}
 Since $n\geq2k>0$, Weyl's invariant theory implies that $L$ is a linear combination of scalars of the form~\eqref{eqn:invariant_thy_basis}.  This has two consequences for us: First, $L\equiv0$ if $g$ is flat.  Second, there is a constant $c\in\bR$ such that $L\equiv c(-\Delta)^{k-1}R$ modulo terms at least quadratic in $\Rm$.  In particular, if $g$ is flat, then
 \[ DL[h]=c(-\Delta)^{k-1}DR[h] = c(-\Delta)^{k-1}\left(\delta^2h-\Delta\tr h\right) \]
 for any $h\in T_g\kM$.  Integrating this against $\dvol$ yields $S=0$, while restricting to divergence-free $h$ yields~\eqref{eqn:rigid_implies_linear_term}.

 Suppose that $c=0$.  Then $DL$ vanishes identically and $L$ is a linear combination of scalars of the form
 \[ \mathrm{contr}\left(\nabla^{2j}\Rm \otimes \nabla^{2k-2j-4}\Rm\right) , \]
 $0\leq j\leq k-2$, modulo terms at least third-order in $\Rm$.  Since $g$ is flat, we have that (see~\cite{Besse})
 \begin{equation}
  \label{eqn:linearize_Rm_flat}
  DR_{ijk}{}^l[h] = -\frac{1}{2}\left(\nabla_i\nabla_kh_j{}^l + \nabla_j\nabla^lh_{ik} - \nabla_j\nabla_kh_i{}^l - \nabla_i\nabla^lh_{jk}\right)
 \end{equation}
 at $g$.  Let $\Upsilon\in C^\infty(M)$ and set $h=E(\Upsilon)$ for $E$ as in~\eqref{eqn:decomp_ker_delta}.  From~\eqref{eqn:linearize_Rm_flat} we readily conclude that
 \[ D\Rm[h] = \frac{n}{2(n-1)}D\Rm(\Upsilon) . \]
 Recalling that $L$ has no terms which are linear in $\Rm$, we conclude that
 \[ D^2L[h,h] = \left(\frac{n}{2(n-1)}\right)^2D^2L(\Upsilon,\Upsilon) . \]
 In particular, since $L$ is infinitesimally rigid at flat metrics, we have that
 \begin{equation}
  \label{eqn:inf_rigid_contradiction}
  \int_M D^2L(\Upsilon,\Upsilon)\,\dvol \leq 0
 \end{equation}
 with equality if and only if $\Upsilon$ is constant.

 We now derive a contradiction to~\eqref{eqn:inf_rigid_contradiction}.  Recall from~\eqref{eqn:cvi_linearization} that
 \[ DL(\Upsilon) = -2kL\Upsilon - \delta\left(T(d\Upsilon)\right) . \]
 The fact that $L$ contains no terms which are linear in $\Rm$ implies that $T\equiv0$ at flat metrics.  It thus follows $D^2L\colon S^2T_g\kC\to C^\infty(M)$ is a divergence if $g$ is flat.  In particular,
 \[ \int_M D^2L(\Upsilon,\Upsilon)\,\dvol_g = 0 \]
 for all $\Upsilon\in C^\infty(M)$.  This contradicts the characterization of equality in~\eqref{eqn:inf_rigid_contradiction}.  Hence $c\not=0$, as desired.
\end{proof}

The results of this section illustrate the importance of the second metric variation $D^2L\in S^2T^\ast\kM$ of a CVI $L$.  While it is straightforward to compute $D^2L$ from $L$, in practice one only needs an integral formula for $D^2L$ at an $L$-critical metric.  The following observation should prove useful in deriving such formulae in general.

\begin{lem}
 \label{lem:D2L-for-L-singular}
 Let $L$ be a CVI and let $(M^n,g)$ be an $L$-singular manifold with $f$ a nontrivial element of $\ker\Gamma^\ast$.  Then
 \[ \int_M f\,D^2L[h,h]\,\dvol = \int_M \lp D\left(\Gamma^\ast(f)\right)[h],h\rp\,\dvol - \frac{1}{2}\int_M\lp \Gamma^\ast(f\tr h),h\rp\,\dvol \]
 for all $h\in T_g\kM\cong S^2T^\ast M$.
\end{lem}

\begin{proof}
 Define $\mF\colon\kM\to\bR$ by
 \[ \mF(\og) = \int_M fL(\og)\,\dvol_g . \]
 Since $\Gamma^\ast(f)=0$, we see that
 \[ \left.\frac{d}{dt}\right|_{t=0}\mF(g+th) = \int_M f\,DL[h]\,\dvol_g = \int_M \lp \Gamma^\ast(f),h\rp\,\dvol_g = 0 \]
 for all $h\in T_g\kM$.  In particular, $g$ is a critical point of $\mF$.  It follows that
 \begin{equation}
  \label{eqn:D2mF}
  \left.\frac{d^2}{dt^2}\right|_{t=0}\mF(g+th) = \int_M f\,D^2L[h,h]\,\dvol_g .
 \end{equation}

 We now compute the second derivative of $\mF$ in another way.  Let $t\mapsto g_t$ be a smooth curve with $g_0=g$.  Differentiating the identity
 \[ \int_M \lp \Gamma_{g_t}^\ast(f), h\rp_{g_t} \,\dvol_{g_t} = \int_M f\,DL_{g_t}[h]\,\dvol_{g_t} \]
 at $t=0$ yields
 \begin{multline*}
  \int_M f\,D^2L[h,h]\,\dvol = \int_M \lp D\left(\Gamma^\ast(f)\right)[h],h\rp\,\dvol - \frac{1}{2}\int_M\lp \Gamma^\ast(f\tr h),h\rp\,\dvol \\ + \frac{1}{2}\int_M\lp \Gamma^\ast(f), (\tr h)h\rp\,\dvol - 2\int_M \lp \Gamma^\ast(f), h\circ h\rp\,\dvol ,
 \end{multline*}
 where $h\circ h$ is the composition of $h$ with itself, regarded as a section of $\End(TM)$ via $g$, and all quantities, including the Riemannian volume forms, are computed at $g$.  Recalling that $\Gamma^\ast(f)=0$ and inserting the above into~\eqref{eqn:D2mF} yields the desired result.
\end{proof}

\section{Examples of CVIs}
\label{sec:weight}

We conclude by discussing examples of CVIs through two families.  First, in Subsection~\ref{subsec:std} we describe Branson's $Q$-curvatures~\cite{Branson1995,FeffermanGraham2012} and the renormalized volume coefficients~\cite{Graham2000} from the point of view of CVIs, noting in particular that they provide two families of CVIs of weight $-2k$, $k\in\bN$, which are positive and variationally stable at any Einstein metric with positive scalar curvature.  Second, we give bases for the spaces of CVIs of weight $-2k$, $k\in\{1,2,3\}$, on Riemannian manifolds of dimension $n\geq 2k$, and illustrate the relative ease with which one can check which such CVIs are variationally stable at any Einstein metric with positive scalar curvature.

We express our bases for the spaces of CVIs of weight at least $-6$ in terms of covariant derivatives of the tensors
\begin{align*}
 J & = \frac{1}{2(n-1)}R, \\
 P_{ij} & = \frac{1}{n-2}\left(R_{ij} - Jg_{ij}\right), \\
 C_{ijk} & = \nabla_iP_{jk} - \nabla_jP_{ik}, \\
 W_{ijkl} & = R_{ijkl} - P_{ik}g_{jl} - P_{jl}g_{ik} + P_{il}g_{jk} + P_{jk}g_{il}, \\
 B_{ij} & = \nabla^k C_{kij} + W_{isjt}P^{st}
\end{align*}
on a Riemannian $n$-manifold, where $R_{ijkl}$ is the Riemannian curvature tensor
\[ \nabla_i\nabla_j X^k - \nabla_j\nabla_i X^k = R_{ij}{}^k{}_s X^s . \]
Note that $P_{ij}$ is the \emph{Schouten tensor}, $W_{ijkl}$ is the \emph{Weyl tensor}, $C_{ijk}$ is the \emph{Cotton tensor}, and $B_{ij}$ is the \emph{Bach tensor}.  Our derivation requires the divergence formulae
\begin{align*}
 \nabla^kW_{ijkl} & = (n-3)C_{ijl}, \\
 \nabla^kB_{jk} & = -(n-4)C_{sjt}P^{st} .
\end{align*}
We also require the conformal linearizations of these tensors.  Recall the well-known variational formulae (e.g.\ \cite[Chapter~1]{Besse})
\begin{align*}
 DP_{ij}(\Upsilon) & = -\Upsilon_{ij}, \\
 D\nabla_i\omega_j(\Upsilon) & = -\Upsilon_i\omega_j - \omega_i\Upsilon_j + \Upsilon^k\omega_k g_{ij}
\end{align*}
for all one-forms $\omega_j$.  From these formulae it is straightforward to compute that
\begin{align*}
 DC_{ijk}(\Upsilon) & = W_{ij}{}^s{}_k\Upsilon_s, \\
 DB_{ij}(\Upsilon) & = -2\Upsilon B_{ij} - (n-4)\Upsilon^s\left(C_{isj}+C_{jsi}\right) .
\end{align*}

\subsection{$Q$-curvatures and renormalized volume coefficients}
\label{subsec:std}

Given $k\in\bN$, the $k$-th renormalized volume coefficients $v_k$ and the $Q$-curvatures $Q_{2k}$ of order $2k$ provide important examples of CVIs defined on any Riemannian manifold $(M^n,g)$ of dimension $n\geq 2k$ which are elliptic and variationally stable at any Einstein metric with positive scalar curvature.  One way to derive these properties is through their definitions in terms of Poincar\'e metrics, as we briefly summarize below.

A \emph{Poincar\'e manifold} is a triple $(X^{n+1},M,g_+)$ consisting of a complete Riemannian manifold $(X^{n+1},g_+)$ such that $X$ is (diffeomorphic to) the interior of a compact manifold $\oX$ with boundary $M=\partial\oX$ and for which there is a defining function $r\colon\oX\to[0,\infty)$ such that
\begin{enumerate}
 \item $M=r^{-1}(0)$ and $dr$ is nowhere vanishing along $M$,
 \item $r^2g_+$ extends to a $C^{n-1,\alpha}$ metric on $\oX$, and
 \item the metric $g_+$ is asymptotically Einstein:
 \begin{align*}
  \Ric(g_+) + ng_+ & = O(r^{n-2}), \\
  \tr_h r^{2-n}i^\ast\left(\Ric(g_+)+ng_+\right) & = 0 .
 \end{align*}
\end{enumerate}
where $i\colon M\to\oX$ is the inclusion mapping.  If $r$ is a defining function as above and $\Upsilon\in C^\infty(\oX)$, then $e^\Upsilon r$ is also a defining function.  Thus only the \emph{conformal infinity} $(M,[i^\ast(r^2g_+)])$ is determined by a Poincar\'e manifold.  Conversely, Fefferman and Graham showed~\cite{FeffermanGraham2012} that given any conformal manifold $(M^n,[h])$, there exists a Poincar\'e manifold with that conformal infinity.  Indeed, they showed that the Poincar\'e metric can be put into a canonical local form near the boundary, in the sense that for any representative $h\in[h]$, there is a \emph{geodesic defining function} $r$, asymptotically unique near $M$, such that
\begin{equation}
 \label{eqn:normal_form}
 g_+ = r^{-2}\left(dr^2 + h_r\right)
\end{equation}
in a neighborhood $[0,\varepsilon)\times M$ of $M$, where $h_r$ is a one-parameter family of metrics on $M\cong\{r\}\times M$ such that $h_0=h$.  Moreover, the metric $g_+$ is uniquely determined to order $O(r^n)$ by $h$; indeed,
\begin{equation}
 \label{eqn:normal_form_asymptotics}
 h_r = \begin{cases}
        h + r^2h_{(2)} + \dotsb + r^{n-1}h_{(n-1)} + Hr^n + o(r^n), & \text{$n$ odd}, \\
        h + r^2h_{(2)} + \dotsb + r^{n-2}h_{(n-2)} + h_{(n)}r^n\log r + Hr^n + o(r^n), & \text{$n$ even},
 \end{cases}
\end{equation}
with each of the terms $h_{(j)}$, $j\leq n$, a natural Riemannian tensor determined only by $h$ and vanishing if $j$ is odd.  Additionally, the term $h_{(n)}$ is trace-free.  The term $H$ is not locally determined, but its trace $\tr_h H$ is a natural Riemannian invariant determined only by $h$.  Moreover, if $n$ is odd, then $\tr_hH=0$.

For the remainder of this subsection, we let $(M^n,h)$ be a Riemannian manifold and associate to it a Poincar\'e manifold $(X^{n+1},M^n,g_+)$ and a geodesic defining function $r$.

To define the renormalized volume coefficients, let $\dvol_{g_+}$ be the Riemannian volume element of $g_+$.  It follows from~\eqref{eqn:normal_form_asymptotics} that near $M$, the asymptotic expansion
\[ \left(\frac{\det h_r}{\det h}\right)^{1/2} = \sum_{k=0}^{\lfloor n/2\rfloor} (-2)^{-k}v_{k}r^{2k} + O(r^{n+1}) \]
contains only even terms $v_k$ until order $o(r^{n+1})$.  Moreover, the terms $v_k$, $k\leq n/2$, are complete contractions of polynomials in $h,h_{(j)},\tr_hH$, $2\leq j<n$ even, and hence are natural Riemannian invariants.  For each $k\in\bN$ with $k\leq n/2$, we call $v_k$ the \emph{$k$-th renormalized volume coefficient}.  Clearly $v_0=1$, and our normalization is such that $v_1=\sigma_1=J$ and $v_2=\sigma_2=\frac{1}{2}(J^2-\lv P\rv^2)$; see~\cite{Graham2009}.  Graham showed~\cite[Theorem~1.5]{Graham2009} (see also~\cite{ISTY2000}) that the renormalized volume coefficients are CVIs for all $k\leq n/2$ by directly computing the conformal linearization of $v_k$.  In particular, using the fact that if $(M^n,h)$ is Einstein, then the normal form~\eqref{eqn:normal_form} of the Poincar\'e metric $g_+$ is
\begin{equation}
 \label{eqn:einstein_normal_form}
 g_+ = r^{-2}\left(dr^2 + h - r^2P + \frac{1}{4}P^2 r^4\right),
\end{equation}
one can conclude that
\begin{equation}
 \label{eqn:conformal_linearization_vk_einstein}
 Dv_k(\Upsilon) = -2kv_k\Upsilon - \delta\left(T_{k-1}(\nabla\Upsilon)\right)
\end{equation}
at any Einstein metric $(M^n,h)$, where $T_{k-1}$ is the $(k-1)$-th Newton tensor; see~\cite{Viaclovsky2000}.  In this way, Chang, Fang and Graham~\cite{ChangFangGraham2012} showed that the renormalized volume coefficients are elliptic and variationally stable at any Einstein manifold with positive scalar curvature.

To define the $Q$-curvatures, we first recall the construction of the GJMS operators~\cite{GJMS1992} via Poincar\'e metrics.  Given $f\in C^\infty(M)$, Graham and Zworski showed~\cite{GrahamZworski2003} that there is a function $u\in C^\infty(X)$ which is a solution to the equation
\begin{equation}
 \label{eqn:poisson}
 -\Delta_{g_+}u-\left(\frac{n^2}{4}-k^2\right)u = 0
\end{equation}
of the form $u=Fr^{\frac{n-2k}{2}}+Gr^{\frac{n+2k}{2}}\log r$ with $F,G\in C^\infty(\oX)$ and $F\rv_M=1$.  Moreover, $F\mod O(r^{2k})$ and $G\rv_M$ are determined by $h$ and $f$.  In particular, one may define the operator $P_{2k}$ by
\begin{equation}
 \label{eqn:defn_gjms}
 P_{2k}f := c_kG\rv_M , \qquad c_k = (-1)^{k-1}2^{2k-1}k!(k-1)! .
\end{equation}
With this normalization, $P_{2k}$ is a formally self-adjoint conformally covariant operator with leading order term $(-\Delta)^k$.  Indeed, $P_{2k}$ is the GJMS operator~\cite{GJMS1992} of order $2k$.  Note that $u=1$ is clearly a solution to~\eqref{eqn:poisson} when $k=n/2$, and hence $P_n(1)=0$ whenever $n$ is even.

When $k<n/2$, the \emph{$Q$-curvature of order $2k$}, is $Q_{2k}:=\frac{2}{n-2k}P_{2k}(1)$.  Thus $Q_{2k}$ is a natural Riemannian scalar invariant.  The conformal covariance of $P_{2k}$ implies the conformal transformation law
\begin{equation}
 \label{eqn:conf_noncritQ}
 e^{2k\Upsilon}\hQ_{2k} = \frac{2}{n-2k}e^{-\frac{n-2k}{2}\Upsilon}P_{2k}\left(e^{\frac{n-2k}{2}\Upsilon}\right)
\end{equation}
for $\hg=e^{2\Upsilon}g$.  When $n$ is even, Fefferman and Graham~\cite{FeffermanGraham2002} showed that there is a formal solution $u\in C^\infty(X)$ to the equation
\[ -\Delta_{g_+} u = n \]
of the form $u=\log r + A + Br^n\log r$ with $A,B\in C^\infty(\oX)$ and $A\rv_M=0$.  Moreover, $A\mod O(r^n)$ and $B\rv_M$ are determined by $h$.  The \emph{critical $Q$-curvature} is
\[ Q_n := c_{n/2}B\rv_M \]
for $c_{n/2}$ as in~\eqref{eqn:defn_gjms}.  It is apparent from the construction that $Q_n$ is a natural Riemannian scalar invariant.  Comparison with the definition~\eqref{eqn:defn_gjms} of $P_n$ leads to the conformal transformation law
\begin{equation}
 \label{eqn:conf_critQ}
 e^{n\Upsilon}\hQ_n = Q_n + P_n(\Upsilon)
\end{equation}
for $\hg=e^{2\Upsilon}g$.

Combining~\eqref{eqn:conf_noncritQ} and~\eqref{eqn:conf_critQ} yields the conformal linearization
\begin{equation}
 \label{eqn:Q_cvi}
 DQ_{2k}(\Upsilon) = -2kQ_{2k}\Upsilon + \left(P_{2k}\right)_0(\Upsilon)
\end{equation}
for all $k\leq n/2$, where $\left(P_{2k}\right)_0$ is the GJMS operator $P_{2k}$ with the constant term removed; i.e.\ $\left(P_{2k}\right)_0(\Upsilon):=P_{2k}(\Upsilon)-\Upsilon P_{2k}(1)$.  This recovers the well-known fact that the $Q$-curvatures are elliptic CVIs.  Moreover, if $(M^n,h)$ is an Einstein metric, then the GJMS operators factor~\cite{FeffermanGraham2012,Gover2006q} as products of Schr\"odinger operators,
\begin{equation}
 \label{eqn:GJMS_factor}
 P_{2k} = \prod_{j=1}^k \left(-\Delta + \frac{(n+2j-2)(n-2j)}{2n}J\right) .
\end{equation}
Combining~\eqref{eqn:Q_cvi} and~\eqref{eqn:GJMS_factor}, one sees that the $Q$-curvatures are positive and variationally stable at Einstein metrics with positive scalar curvature.  In fact, it is straightforward to check that at such a metric, $DQ_{2k}$ can be written as a composition $(-\Delta-2J)\circ F$ for $F$ a positive operator which commutes with $(-\Delta-2J)$.

\subsection{CVIs of weight $-2$}
\label{subsec:weight2}

Up to scaling, Weyl's theorem implies that $J$ is the only natural Riemannian invariant of weight $-2$.  Recall from~\eqref{eqn:conformal_linearization_R} that
\[ DJ(\Upsilon) = -2J\Upsilon - \Delta\Upsilon . \]
Thus $J$ is elliptic.  From the Lichernowicz--Obata Theorem, we also see that $J$ is variationally stable at any Einstein metric.

\subsection{CVIs of weight $-4$}
\label{subsec:weight4}

An important step in the computation of functional determinants of powers of conformally covariant operators by Branson and {\O}rsted~\cite{BransonOrsted1991b} is the identification of the space of CVIs of weight $-4$.  This identification is readily made as follows: It is clear that the set
\[ \{ -\Delta J, J^2, \lv P\rv^2, \lv W\rv^2 \} \]
forms a basis for the space of Riemannian invariants of weight $-4$.  The conformal invariance of the Weyl curvature implies that $\lv W\rv^2$ is a conformal invariant of weight $-4$, while direct computations yield
\begin{align*}
 \frac{1}{2}D\Bigl(\int_M J^2\,\dvol\Bigr)(\Upsilon) & = \int_M \left(-\Delta J + \frac{n-4}{2}J^2\right)\Upsilon\,\dvol, \\
 \frac{1}{2}D\Bigl(\int_M \lv P\rv^2\,\dvol\Bigr)(\Upsilon) & = \int_M \left( -\Delta J + \frac{n-4}{2}\lv P\rv^2\right)\Upsilon\,\dvol .
\end{align*}
In particular, we recover the facts that
\begin{align*}
 \sigma_2 & := \frac{1}{2}\left(J^2-\lv P\rv^2\right), \\
 Q_4 & := -\Delta J - 2\lv P\rv^2 + \frac{n}{2}J^2
\end{align*}
are conformally variational and the observation of Branson and {\O}rsted that
\[ \left\{ \sigma_2, Q_4, \lv W\rv^2 \right\} \]
is a basis for the space of CVIs of weight $-4$.

We now consider variational stability of CVIs of weight $-4$.  To that end, we restrict our attention to the span of $\sigma_2$ and $Q_4$.  Indeed, $\lv W\rv^2$ need not be constant at all Einstein metrics.  For example, consider the Page metric~\cite{Page1979} on $\bCP^2\hash\overline{\bCP^2}$, written in coordinates as
\[ g = \alpha^2\,dr^2 + \beta^2\left(d\tau-4\sin^2\bigl(\frac{\rho}{2}\bigr)d\theta^2\right)^2 + \gamma^2\left(d\rho^2+\sin^2\rho\,d\theta^2\right), \]
where $\tau,\rho,\theta\in(0,2\pi)$ and $r\in(0,\nu)$ for $\nu\in(0,1)$ the root of $\nu^4+4\nu^3-6\nu^2+12\nu-3$, and where $\alpha,\beta,\gamma$ are the functions of $r$ satisfying $\alpha\beta=c:=(3+6\nu^2-\nu^4)^{-1}$, $\gamma^2=c(1-r^2)$, and
\[ \beta^2 = c^2(1-r^2)^{-1}\left(\nu^2-r^2\right)\left(3-\nu^2-(1+\nu^2)r^2\right) . \]
It is readily computed that $\Ric(g)=3(1+\nu^2)g$.  With respect to the orthonormal frame $\he_0=\alpha\,dr$, $\he_1=\beta\,\bigl(d\tau-4\sin^2(\frac{\rho}{2})\bigr)$, $\he_i=\gamma e_i$, $i\in\{2,3\}$, where $e_2,e_3$ is a local orthonormal frame for the round two-sphere $(S^2,d\rho^2+\sin^2\rho\,d\theta^2)$, a straightforward but tedious computation yields
\begin{equation}
 \label{eqn:page_weyl}
 \begin{split}
  W_{0101} & = \gamma^{-4}\left(\gamma^2-(1+\nu)^2\gamma^4 - (3+r^2)\beta^2\right), \\
  W_{0123} & = 2\gamma^4\left(c^{-1}\gamma^2\beta\beta^\prime + r\beta^2\right),
 \end{split}
\end{equation}
as well as the identities $W_{0202}=W_{0303}=W_{1212}=W_{1313}$ and $W_{0231}=W_{0312}$.  All remaining nonvanishing components of the Weyl tensor can be determined from these facts via the usual symmetries of the Weyl tensor.  It follows that, by taking isometric products with suitably normalized spheres, one obtains examples of Einstein metrics of all dimensions greater than or equal to $4$ for which $\lv W\rv^2$ is nonconstant; it is an $\bR$-linear combination of a constant and terms quadratic in~\eqref{eqn:page_weyl}.

As discussed in Subsection~\ref{subsec:std}, both $\sigma_2$ and $Q_4$ are variationally stable at any Einstein metric with positive scalar curvature.  To check which elements of $\lp \sigma_2, Q_4\rp$ also satisfy this property, we use~\eqref{eqn:conformal_linearization_vk_einstein}, \eqref{eqn:Q_cvi}, and~\eqref{eqn:GJMS_factor} to compute that
\begin{align*}
 D\sigma_2(\Upsilon) & = \frac{n-1}{n}J\bigl(-\Delta-2J\bigr)\Upsilon , \\
 DQ_4(\Upsilon) & = \bigl(-\Delta+\frac{n^2-4}{n}J\bigr)\bigl(-\Delta-2J\bigr)\Upsilon
\end{align*}
for any Einstein manifold $(M^n,g)$.  Consider the three cones
\begin{align*}
 \mE & := \left\{ \alpha Q_4 + \beta\sigma_2 \suchthat \alpha>0\right\} \cup \left\{ \beta\sigma_2 \suchthat \beta>0 \right\}, \\
 \mV & := \left\{ \alpha Q_4 + \beta\sigma_2 \suchthat \alpha\geq 0, (n^2+2n-4)\alpha+(n-1)\beta\geq0, \alpha^2+\beta^2\not=0\right\}, \\
 \mS\mV & := \left\{ \alpha Q_4 + \beta\sigma_2 \suchthat \alpha\geq 0, (n^2-4)\alpha+(n-1)\beta>0\right\} .
\end{align*}
It readily follows from Corollary~\ref{cor:second_variation} that
\begin{enumerate}
 \item $\mE$ is the cone of CVIs of weight $-4$ which are elliptic with spectrum tending to $+\infty$ at any Einstein metric with positive scalar curvature;
 \item $\mV$ is the cone of CVIs of weight $-4$ which are variationally stable at any Einstein metric with positive scalar curvature;
 \item $\mS\mV$ is the cone of CVIs $L$ of weight $-4$ such that for any Einstein metric $g$ of positive scalar curvature, $L_g>0$ and $DL_g=(-\Delta-2J)\circ F$ for $F$ a positive operator which commutes with $-\Delta-2J$.
\end{enumerate}
Note in particular that
\[ \mS\mV \subsetneq \mV \subsetneq \mE . \]

As noted in Subsection~\ref{subsec:std}, the fourth-order $Q$-curvature is an element of $\mS\mV$.  This cone seems relevant to the study of the functionals $\mF_{\gamma_2,\gamma_3}:=\gamma_2II + \gamma_3III$ on $\kC_1$, the volume-normalized conformal class of a round metric $g_0$ of constant sectional curvature on $S^4$, where $\gamma_2,\gamma_3\in\bR$ and
\begin{align*}
 II[e^{2w}g_0] & = \frac{1}{2}\int_M w\,P_4w\,\dvol_{g_0} + \int_M Q_4w\,\dvol_{g_0}, \\
 III[e^{2w}g_0] & = \frac{1}{2}\int_M \left(J_{e^{2w}g_0}\right)^2\dvol_{e^{2w}g_0}
\end{align*}
are conformal primitives of $Q_4$ and $-\Delta J=Q_4-4\sigma_2$, respectively.  Note that the functional $\mF_{\gamma_2,\gamma_3}$ is the logarithm of the ratio of functional determinants of powers of conformally covariant operators on the four-sphere~\cite{Branson1996,BransonOrsted1991b}.  Let
\[ L_{\gamma_2,\gamma_3} = (\gamma_2+\gamma_3)Q_4 - 4\gamma_3\sigma_2 \]
be the conformal gradient of $\mF_{\gamma_2,\gamma_3}$.  Note that
\begin{enumerate}
 \item $L_{\gamma_2,\gamma_3}\in\mV$ if and only if $5\gamma_2+2\gamma_3\geq0$ and $\gamma_2+\gamma_3\geq0$, with $(\gamma_2,\gamma_3)\not=0$;
 \item $L_{\gamma_2,\gamma_3}\in\mS\mV$ if and only if $\gamma_2>0$ and $\gamma_2+\gamma_3\geq0$.
\end{enumerate}
On the other hand, it is known that
\begin{enumerate}
 \item if $\gamma_2\gamma_3\geq0$ with at least one of $\gamma_2,\gamma_3$ nonzero, then $\mF_{\gamma_2,\gamma_3}$ is extremized exactly by the round metrics in $\kC_0$~\cite{BransonChangYang1992};
 \item if $\gamma_2\gamma_3<0$ and $\gamma_2/\gamma_3\in(-9/4,-1/4)$, then $\mF_{\gamma_2,\gamma_3}$ is neither bounded above nor bounded below in $\kC_0$~\cite{GurskyMalchiodi2012}; and
 \item if $\gamma_2\gamma_3<0$ and $\gamma_2/\gamma_3\in(-1,-1/4)$, then $\mF_{\gamma_2,\gamma_3}$ admits a critical point in $\kC_0$ which is not round~\cite{GurskyMalchiodi2012}.
\end{enumerate}
In particular, there are elements $L_{\gamma_2,\gamma_3}\in\mV\setminus\mS\mV$ --- including the conformal gradient of the functional determinant of the Paneitz operator~\cite{Branson1996} --- for which the corresponding functional $\mF_{\gamma_2,\gamma_3}$ has non-round critical points in $\kC_0$.  It is not currently known whether there are elements $L_{\gamma_2,\gamma_3}\in\mS\mV$ which satisfy this property, though there are elements $L_{\gamma_2,\gamma_3}\in\mS\mV$ for which the corresponding functional $\mF_{\gamma_2,\gamma_3}$ is neither bounded above nor bounded below.

\subsection{CVIs of weight $-6$}
\label{subsec:weight6}

It is known~\cite{CapGover2003,FeffermanGraham2012} that
\begin{align*}
 L_1 & := W_{ij}{}^{kl}W_{kl}{}^{st}W_{st}{}^{ij}, \\
 L_2 & := W_i{}^k{}_j{}^l W_k{}^s{}_l{}^t W_s{}^i{}_t{}^j, \\
 L_3 & := -\frac{1}{2}\Delta\lv W\rv^2 - 2(n-10)\nabla^l\left(W_{ijkl}C^{ijk}\right) + 2(n-10)P_i^sW_{sjkl}W^{ijkl} \\
  & \quad + 2J\lv W\rv^2 - 2(n-5)(n-10)\lv C\rv^2
\end{align*}
form a basis for the set of natural pointwise conformal invariants of weight $-6$.  A basis for the space of natural Riemannian invariants of weight $-6$ is given by the following lemma.  Note that, in the notation of Section~\ref{sec:bg}, this basis shows that $\mR_6^n/\im\delta$ is $10$-dimensional.

\begin{lem}
 \label{lem:weight6_riem_basis}
 The set
 \begin{multline*}
  \bigl\{ J^3, J\lv P\rv^2, \tr P^3, \lp B,P\rp, -J\Delta J, W_{ijkl}P^{ik}P^{jl}, J\lv W\rv^2, L_1, L_2, L_3, \\ -\Delta J^2, -\Delta\lv P\rv^2, -\Delta\lv W\rv^2, \delta\left(P(\nabla J)\right), \nabla^i\left(C_{sit}P^{st}\right), \nabla^l\left(W_{ijkl}C^{ijk}\right), \Delta^2J \bigr\}
 \end{multline*}
 forms a basis for the space $\mR_6^n$ of natural Riemannian invariants of weight $-6$.
\end{lem}

\begin{proof}
 Recall the identity
 \[ \Delta P_{ij} = B_{ij} + \nabla_i\nabla_j J - 2W_{isjt}P^{st} + nP_{is}P_j^s - \lv P\rv^2g_{ij} . \]
 Direct computations yield
 \begin{align*}
  \frac{1}{2}\Delta J^2 & = \lv\nabla J\rv^2 + J\Delta J, \\
  \delta\left(P(\nabla J)\right) & = \lp P,\nabla^2 J\rp + \lv\nabla J\rv^2, \\
  \frac{1}{2}\Delta\lv P\rv^2 & = \lv\nabla P\rv^2 + \lp B,P\rp + \lp P,\nabla^2J\rp - 2W_{ijkl}P^{ik}P^{jl} + n\tr P^3 - J\lv P\rv^2, \\
  \nabla^i\left(C_{sit}P^{st}\right) & = -\frac{1}{2}\lv C\rv^2 - \lp B,P\rp + W_{ijkl}P^{ik}P^{jl} , \\
  \nabla^l\left(W_{ijkl}C^{ijk}\right) & = -W_{ijkl}\nabla^iC^{klj} - (n-3)\lv C\rv^2 .
 \end{align*}
 Additionally, the Bianchi identity $\nabla_mR_{ijkl}+\nabla_iR_{jmkl}+\nabla_jR_{mikl}=0$ implies that
 \begin{multline}
  \label{eqn:Weyl_bianchi}
  \nabla_mW_{ijkl}+\nabla_iW_{jmkl}+\nabla_jW_{mikl} \\ = C_{imk}g_{jl} + C_{mjk}g_{il} + C_{jik}g_{ml} - C_{iml}g_{jk} - C_{mjl}g_{ik} - C_{jil}g_{mk} .
 \end{multline}
 As a consequence, we see that
 \begin{multline*}
  \frac{1}{2}\Delta\lv W\rv^2 = \lv\nabla W\rv^2 - 2(n-2)\nabla^l\left(W_{ijkl}C^{ijk}\right) + 2J\lv W\rv^2 + 2(n-2)P_i^sW_{sjkl}W^{ijkl} \\ - 2(n-2)(n-3)\lv C\rv^2 - L_1 - 4L_2 .
 \end{multline*}
 Combining these facts shows that our set spans $\mR_6^n$.  Finally, it follows readily from Weyl's invariant theory (e.g.\ \cite{Donnelly1974}) that $\dim\mR_6^n=17$ if $n\geq6$, and hence our set is also a basis.
\end{proof}

We next compute the conformal gradients of the integrals of the basis elements from Lemma~\ref{lem:weight6_riem_basis} which are neither divergences nor pointwise conformal invariants.

\begin{lem}
 \label{lem:conformal_gradients_weight6}
 Let $(M^n,g)$ be a Riemannian manifold.  As functionals on the conformal class $[g]$, it holds that
 \begin{align*}
  D\mS_{J^3}(\Upsilon) & = \bigllp -3\Delta J^2 + (n-6)J^3 , \Upsilon \bigrrp , \\
  D\mS_{J\lv P\rv^2}(\Upsilon) & = \bigllp -\Delta\lv P\rv^2 - 2\delta\left(P(\nabla J)\right) - \Delta J^2 + (n-6)J\lv P\rv^2 , \Upsilon \bigrrp , \\
  D\mS_{\tr P^3}(\Upsilon) & = \bigllp -3\delta\left(P(\nabla J)\right) - \frac{3}{2}\Delta\lv P\rv^2 - 3\nabla^k\left(C_{skt}P^{st}\right) + (n-6)\tr P^3 , \Upsilon \bigrrp , \\
  D\mS_{-J\Delta J}(\Upsilon) & = \bigllp 2\Delta^2J + \frac{n-2}{2}\Delta J^2 - (n-6)J\Delta J , \Upsilon \bigrrp , \\
  D\mS_{\lp B,P\rp}(\Upsilon) & = \bigllp 3(n-4)\nabla^k\left(C_{skt}P^{st}\right) + (n-6)\lp B,P\rp , \Upsilon \bigrrp , \\
  D\mS_{W\cdot P^2}(\Upsilon) & = \bigllp 2(n-3)\nabla^k(C_{skt}P^{st}) + \nabla^l(W_{ijkl}C^{ijk}) + (n-6)W\cdot P^2 , \Upsilon \bigrrp , \\
  D\mS_{J\lv W\rv^2}(\Upsilon) & = \bigllp -\Delta\lv W\rv^2 + (n-6)J\lv W\rv^2 , \Upsilon \bigrrp ,
 \end{align*}
 where we denote $W\cdot P^2:=W_{ijkl}P^{ik}P^{jl}$.
\end{lem}

\begin{proof}
 Observe that
 \begin{align*}
  \nabla^j\nabla^k\left(P_j^sP_{sk}\right) & = \delta\left(P(\nabla J)\right) + \frac{1}{2}\Delta\lv P\rv^2 + \nabla^k\left(C_{skt}P^{st}\right), \\
  \nabla^j\nabla^kB_{jk} & = -(n-4)\nabla^k\left(C_{skt}P^{st}\right), \\
  \nabla^j\nabla^k\left(W_{jskt}P^{st}\right) & = -(n-3)\nabla^k\left(C_{skt}P^{st}\right) - \frac{1}{2}\nabla^l\left(W_{ijkl}C^{ijk}\right), \\
  \nabla^i\nabla^s\left(W_{ijkl}W_s{}^{jkl}\right) & = -(n-4)\nabla^l(W_{ijkl}C^{ijk}) + \frac{1}{4}\Delta\lv W\rv^2,
 \end{align*}
 where the last identity uses~\eqref{eqn:Weyl_bianchi}.  The result follows from these identities, the formulae for the conformal linearizations of the Schouten and the Bach tensors, and integration by parts.
\end{proof}

We can now find a basis for the vector space of CVIs of weight $-6$.  To that end, recall that the sixth-order renormalized volume coefficient $v_3$ (see~\cite{Graham2000}) and the sixth-order $Q$-curvature (see~\cite{Branson1985,GoverPeterson2003,Wunsch1986}) are
\begin{align}
 \label{eqn:defn_v3} v_3 & := \frac{1}{6}J^3 - \frac{1}{2}J\lv P\rv^2 + \frac{1}{3}\tr P^3 + \frac{1}{3(n-4)}\lp B, P\rp, \\
 \label{eqn:defn_Q6} Q_6 & := \Delta^2 J - \frac{n-6}{2}J\Delta J - \frac{n+2}{2}\Delta J^2 + 8\delta\left(P(\nabla J)\right) + 4\Delta\lv P\rv^2 \\
  \notag & \quad + \frac{n^2-4}{4}J^3 - 4nJ\lv P\rv^2 + 16\tr P^3 + \frac{16}{n-4}\lp B,P\rp .
\end{align}
Both $v_3$ and $Q_6$ are conformally variational.  We also denote
\begin{align}
 \label{eqn:defn_I1} I_1 & := -\Delta J^2 + \frac{n-6}{3}J^3, \\
 \label{eqn:defn_I2} I_2 & := -\Delta\lv P\rv^2 - 2\delta\left(P(\nabla J)\right) - \Delta J^2 + (n-6)J\lv P\rv^2, \\
 \label{eqn:defn_K1} K_1 & := 3(n-4)\nabla^k\left(C_{skt}P^{st}\right) + (n-6)\lp B,P\rp, \\
 \label{eqn:defn_K2} K_2 & := 2(n-3)\nabla^k(C_{skt}P^{st}) + \nabla^l(W_{ijkl}C^{ijk}) + (n-6)W\cdot P^2, \\
 \label{eqn:defn_K3} K_3 & := -\left\lp W^2 - \frac{1}{4}\lv W\rv^2g,P\right\rp + \frac{n-4}{2}\lv C\rv^2 .
\end{align}

The discussion above establishes that these invariants together with the pointwise conformal invariants $L_1$, $L_2$, and $L_3$ form a basis for the vector space of CVIs of weight $-6$.

\begin{prop}
 \label{prop:cvi6_basis}
 The set $\{v_3,Q_6,I_1,I_2,K_1,K_2,K_3,L_1,L_2,L_3\}$ forms a basis for the vector space of CVIs of weight $-6$.
\end{prop}

Let us discuss each of these basis elements in some detail.

\begin{example}
 \label{ex:v3}
 Let $(M^n,g)$ be an Einstein manifold with $n\geq6$.  Combining the conformal linearization~\eqref{eqn:conformal_linearization_vk_einstein} of $v_3$ with the fact $v_3=\sigma_3$ in this case yields
 \[ Dv_3(\Upsilon) = \frac{(n-1)(n-2)}{2n^2}J^2\left(-\Delta-2J\right)\Upsilon \]
 at $g$.  This recovers the fact that $v_3$ is positive and variationally stable at any Einstein metric with positive scalar curvature.
\end{example}

\begin{example}
 \label{ex:Q6}
 Combining the conformal linearization~\eqref{eqn:Q_cvi} of $Q_6$ and the factorization~\eqref{eqn:GJMS_factor} of $P_6$ at an Einstein metric, we see that if $(M^n,g)$ is Einstein, then the conformal linearization of $Q_6$ at $g$ is
 \[ DQ_6(\Upsilon) = \left(\Delta^2 - \frac{3n^2-2n-32}{2n}J\Delta + \frac{3(n^2-16)(n^2-4)}{2n^2}J^2\right)\left(-\Delta-2J\right)\Upsilon . \]
 This and the formula for $Q_6$ at $g$ deduced by applying~\eqref{eqn:GJMS_factor} to a constant function recovers the fact that $Q_6$ is positive and variational at any Einstein metric of positive scalar curvature.
\end{example}

\begin{example}
 \label{ex:I1}
 It is clear from~\eqref{eqn:defn_I1} that $I_1=\frac{n-6}{3}J^3$ for any Einstein manifold $(M^n,g)$.  For general $n$-manifolds, we readily compute that
 \[ DI_1(\Upsilon) = - 6I_1\Upsilon + 2\Delta\left(J\Delta\Upsilon\right) - (n-10)\delta\left(J^2\,d\Upsilon\right) . \]
 In particular, if $(M^n,g)$ is Einstein, then
 \[ DI_1(\Upsilon) = 2J\left(-\Delta+\frac{n-6}{2}J\right)\left(-\Delta-2J\right)\Upsilon . \]
 Thus $I_1$ is positive and variationally stable at any Einstein metric with positive scalar curvature.
\end{example}

\begin{example}
 \label{ex:I2}
 It is clear from~\eqref{eqn:defn_I2} that $I_2=\frac{n-6}{n}J^3$ for any Einstein manifold $(M^n,g)$.  For general $n$-manifolds, we readily compute that
 \begin{multline*}
  DI_2(\Upsilon) = -6I_2\Upsilon + 2\Delta\left(J\Delta\Upsilon\right) + 2\Delta\delta\left(P(\nabla\Upsilon)\right) + 2\delta\left(P(\nabla\Delta\Upsilon)\right) \\ - \delta\left(\left((n-10)\lv P\rv^2g - 4J^2g + 2\nabla^2J + 2(n-8)JP\right)(\nabla\Upsilon)\right) .
 \end{multline*}
 In particular, if $(M^n,g)$ is Einstein, then
 \[ DI_2(\Upsilon) = \frac{2(n+2)}{n}J\left(-\Delta+\frac{3(n-6)}{2(n+2)}J\right)\left(-\Delta-2J\right)\Upsilon . \]
 Thus $I_2$ is positive and variationally stable at any Einstein metric with positive scalar curvature.
\end{example}

\begin{example}
 \label{ex:K1}
 It is clear from~\eqref{eqn:defn_K1} that $K_1=0$ for any Einstein manifold $(M^n,g)$.  For general $n$-manifolds, we readily compute that
 \[ DK_1(\Upsilon) = -6K_1\Upsilon - \delta\left(\left(6(n-4)W\cdot P - 2(n-3)B\right)(\nabla\Upsilon)\right), \]
 where $(W\cdot P)_{ij}:=W_{isjt}P^{st}$.  In particular, if $(M^n,g)$ is Einstein, then $DK_1(\Upsilon)=0$.  Thus $K_1$ is neither positive nor variationally stable at all Einstein metrics with positive scalar curvature.
\end{example}

\begin{example}
 \label{ex:K2}
 It is clear from~\eqref{eqn:defn_K2} that $K_2=0$ for any Einstein manifold $(M^n,g)$.  For general $n$-manifolds, we readily compute that
 \[ DK_2(\Upsilon) = -6K_2\Upsilon - \delta\left(\left(6(n-4)W\cdot P + W^2 - 2(n-3)B\right)(\nabla\Upsilon)\right), \]
 where $(W^2)_{ij}:=W_{istu}W_j{}^{stu}$.  In particular, if $(M^n,g)$ is Einstein, then
 \[ DK_2(\Upsilon) = -\delta\left(W^2(\nabla\Upsilon)\right) . \]
 Thus $K_2$ is neither positive nor variationally stable at all Einstein metrics with positive scalar curvature.
\end{example}

Of course, we can add multiples of $K_1$ and nonnegative multiples of $K_2$ to any CVI $L$ of weight $-6$ which is positive and variationally stable at any Einstein metric to obtain a new CVI with the same properties.

\begin{example}
 \label{ex:pointwise_conformal_invariants6}
 Each of the pointwise conformal invariants $L_1$, $L_2$, and $L_3$, as well as the CVI $K_3$, are not necessarily constant at an Einstein metric.  This is a consequence of~\eqref{eqn:page_weyl} when computing the Weyl curvature of Einstein products of the Page metric with suitably normalized round spheres.
\end{example}

%Acknowledgments
\subsection*{Acknowledgments}
JSC was supported by a grant from the Simons Foundation (Grant No.\ 524601).  WY was supported by NSFC (Grant No.\ 11521101, No.\ 11601531) and the Fundamental Research Funds for the Central Universities (Grant No.\ 2016-34000-31610258).

\bibliographystyle{abbrv}
\bibliography{../bib}
\end{document}